\documentclass[11pt,letterpaper]{amsart}
\usepackage[centertags]{amsmath}
\usepackage{amsfonts}
\usepackage{amssymb}
\usepackage{amsthm}
\usepackage{newlfont}
\usepackage{amscd}
\usepackage{mathrsfs}
\usepackage[all,cmtip]{xy}
\usepackage{tikz-cd}
\tikzcdset{arrow style=Latin Modern, row sep/normal=1.35em, column sep/normal=1.8em, cramped}
\usepackage{tikz}
\usepackage[pagebackref=false,colorlinks]{hyperref}

\definecolor{mycolor}{HTML}{F7F8E0}
\definecolor{myorange}{RGB}{245,156,74}
\definecolor{cadetgrey}{rgb}{0.57, 0.64, 0.69}
\definecolor{calpolypomonagreen}{rgb}{0.12, 0.3, 0.17}

\hypersetup{pdffitwindow=true,linkcolor=calpolypomonagreen,citecolor=calpolypomonagreen,urlcolor=calpolypomonagreen}
\usepackage{cite}
\usepackage{fancyhdr}
\usepackage{stmaryrd}

\usepackage[OT2,OT1]{fontenc}
\newcommand\cyr{%
\renewcommand\rmdefault{wncyr}%
\renewcommand\sfdefault{wncyss}%
\renewcommand\encodingdefault{OT2}%
\normalfont
\selectfont}
\DeclareTextFontCommand{\textcyr}{\cyr}

\usepackage[toc, page] {appendix}

\makeatletter
\let\@wraptoccontribs\wraptoccontribs
\makeatother

\numberwithin{equation}{section}

\newtheorem{thm}{Theorem}[section]

\newtheorem{cor}[thm]{Corollary}

\newtheorem{prop}[thm]{Proposition}

\newtheorem{assu}[thm]{Assumption}
\newtheorem{conj}[thm]{Conjecture}

\theoremstyle{definition}
\newtheorem{defn}[thm]{Definition}

\newtheorem{caveat}[thm]{Caveat}
\newtheorem{exer}[thm]{Exercise}
\newtheorem{rem}[thm]{Remark}

\newtheorem{ques}[thm]{Question}

\newtheorem{exam}[thm]{Example}

\newcommand{\ks}{\boldsymbol{\kappa}}

\newcommand{\widedelta}{\widetilde{\delta}}

\newcommand{\sha}{\textrm{{\cyr SH}}}

\newcommand{\et}{\mathrm{\acute{e}t}}

\begin{document}
\title{A user's guide to Beilinson--Kato's zeta elements}
\author{Chan-Ho Kim}
\address{
Department of Mathematics and Institute of Pure and Applied Mathematics,
Jeonbuk National University,
567 Baekje-daero, Deokjin-gu, Jeonju, Jeollabuk-do 54896, Republic of Korea
}
\email{chanho.math@gmail.com}
\thanks{This research was partially supported 
by a KIAS Individual Grant (SP054103) via the Center for Mathematical Challenges at Korea Institute for Advanced Study,
by the National Research Foundation of Korea(NRF) grant funded by the Korea government(MSIT) (No. 2018R1C1B6007009, 2019R1A6A1A11051177),
by the International Centre for Theoretical Sciences (ICTS) for the online program - Elliptic curves and the special values of $L$-functions (code: ICTS/ECL2021/8), and
by Global-Learning \& Academic research institution for Master’s$\cdot$Ph.D. Students, and Postdocs (LAMP) Program of the National Research Foundation of Korea (NRF) funded by the Ministry of Education (No. RS-2024-00443714).
}
\date{\today}
\maketitle
\begin{abstract}
In his ground-breaking work \cite{kato-euler-systems}, K. Kato constructed the Euler system of Beilinson--Kato's zeta elements and proved spectacular results on the Iwasawa main conjecture for elliptic curves and the classical and $p$-adic Birch and Swinnerton-Dyer conjectures by using these elements.
The goal of this expository lecture note is to explain how Kato's Euler systems fit into the framework of the arithmetic of elliptic curves and their Iwasawa theory, and we hope that this approach eventually helps the reader to read Kato's original paper \emph{more easily and with less pain}.
% It is written for the proceedings of the program \emph{Elliptic curves and the special values of $L$-functions (ONLINE)} in ICTS in August 2021.
\end{abstract}
\setcounter{tocdepth}{1}
\tableofcontents

\section*{Prologue}
\subsection{}
It does not seem terribly bad to start with my personal experience.
From around 2014, I have started to read Kato's beautiful work on his Euler systems \cite{kato-euler-systems} and have failed so many times due to many reasons, mainly the lack of my mathematical maturity.
I still learn a lot from this paper when I read it again. 
My initial motivation for studying this paper was to develop the mod $p$ numerical criterion for verifying the main conjecture \cite{kks}.
From this perspective, I did not care about the construction of the Euler system too much for the first time, but I tried hard to understand the meaning of the final output of the construction. Even only for such a ``small" purpose, it turns out that I needed to understand the precise meaning of
$${}_{c,d}z^{(p)}_{m}(f, r, r' , \xi , S) \in \mathrm{H}^1(\mathbb{Q}(\zeta_m), T_f(k-r))$$
in \cite[(8.1.3) and Ex. 13.3]{kato-euler-systems} and
$$\mathbf{z}^{(p)}_\gamma  \in \varprojlim_n \mathrm{H}^1(\mathbb{Q}(\zeta_{p^n}), T_f)$$
in \cite[Thm. 12.5]{kato-euler-systems}. 
Ugh, I should go back to the construction to some extent and try to understand the behavior of these elements under Bloch--Kato's dual exponential map at least.
It required a considerable amount of time and energy for me at that time.
My aim is for this lecture note to make diving into Kato's original paper a breeze, minimizing the struggle and maximizing understanding.
\subsection{}
Although \cite{kato-euler-systems} was published in 2004, Kato's work was announced and was accepted by experts in early 1990's.
See \cite{kato-hasse-weil, kato-K2-modular-curves, kato-erl-zeta-values,kato-elliptic-curves-modular-forms} for his announcement in Japanese\footnote{These are all hand-written.}. Also, \cite{scholl-kato, rubin-es-mec} were written for the proceeding of the Durham conference held at 1996.
\subsection{}
There are now various expository papers and books on Kato's work from different viewpoints. See \cite{rubin-es-mec, scholl-kato, colmez-p-adic-BSD, delbourgo-book, wang_kato, ochiai-book-two, ochiai-book-two-translated} for example.

\subsection{}
As prerequisites, we expect that the reader knows some algebraic number theory \cite{neukirch}, Galois cohomology \cite[\S1]{rubin-book}, elliptic curves \cite{silverman} and modular forms \cite{diamond-shurman}.
The most important prerequisite would be a certain amount of experiences to try to read \cite{kato-euler-systems}.
\subsection{}
Of course, we are not able to explain every detail of Kato's paper here.
As we mentioned above, our goal is to explain how Kato's Euler systems apply to the arithmetic of elliptic curves.
We explain relatively small parts of his paper in details \emph{only for achieving this goal}.
In this sense, our exposition can be seen as employing a reverse-engineering approach.
Since it grew out from a \emph{lecture note}, this article contains some \emph{intentional} repetitions. 
We hope this structure make the reader more comfortable.

\section{Lecture 1: Overview and Birch and Swinnerton-Dyer conjecture}
In the first lecture, we illustrate how Kato's Euler systems work in the arithmetic of elliptic curves and to sketch how they are constructed briefly. Our primary focus is elliptic curves, and we will not cover higher weight modular forms in this lecture note.

\subsection{Rational points on elliptic curves}
How do we understand rational points on elliptic curves?
\begin{exam} \label{exam:two-elliptic-curves}
Consider two elliptic curves:
\begin{itemize}
\item $E_1: y^2 = x^3 -x$ (\href{https://www.lmfdb.org/EllipticCurve/Q/32/a/3}{32.a.3})
\item $E_2: y^2 +y = x^3 -x$ (\href{https://www.lmfdb.org/EllipticCurve/Q/37/a/1}{37.a.1})
\end{itemize}
Although the defining equations of these two curves do not look very different, their Diophantine properties are significantly different:
\[
\xymatrix{
E_1(\mathbb{Q}) \simeq \mathbb{Z}/2\mathbb{Z} \times \mathbb{Z}/2\mathbb{Z} , & E_2(\mathbb{Q}) \simeq \mathbb{Z} .
}
\]
How can we detect this difference?
\end{exam}
Let $E$ be an elliptic curve over $\mathbb{Q}$. 
As in Example \ref{exam:two-elliptic-curves}, the human eye cannot see whether $E(\mathbb{Q})$ is infinite or not, but $L$-functions see. This is called Birch and Swinnerton-Dyer conjecture.
In this sense, we observe an important inequality\footnote{I learned this inequality from the lecture series of Masato Kurihara at Postech, January 2012.}
$$\textrm{human eyes} < \textrm{$L$-functions}.$$
The following modularity theorem is essential for the statement of Birch and Swinnerton-Dyer conjecture as well as the applications of Kato's Euler systems to elliptic curves.
\begin{thm}[Wiles, Taylor--Wiles, $\cdots$, Breuil--Conrad--Diamond--Taylor] \label{thm:modularity}
There exists a one-to-one correspondence between
\begin{itemize}
\item $\mathbb{Q}$-isogeny classes of rational elliptic curves of conductor $N$ and
\item weight two cuspidal newforms with rational coefficients of level $\Gamma_0(N)$
\end{itemize}
satisfying $L(E,s) = \sum_{n \geq 1} \dfrac{a_n}{n^s}$ with $\mathrm{Re}(s) > 3/2$
where $f_E = \sum_{n \geq 1} a_n q^n \in S_2(\Gamma_0(N))$.
In particular, $L(E,s)$ is entire.
\end{thm}
It is known that the complex $L$-function $L(E,s)$ is an isogeny invariant.
\begin{proof}
See \cite{wiles, taylor-wiles, bcdt}
\end{proof}
The following conjecture is the most important problems in the arithmetic of elliptic curves \cite{bsd-notes-1,bsd-notes-2,tate-arithmetic-elliptic-curves,goldfeld-bsd}.
\begin{conj}[Birch--Swinnerton-Dyer]
Let $E$ be an elliptic curve over $\mathbb{Q}$.
\begin{enumerate}
\item $\mathrm{rk}_{\mathbb{Z}}E(\mathbb{Q}) = \mathrm{ord}_{s=1}L(E, s) (=r)$.
\item If (1) holds and Tate--Shafarevich group $\sha(E/\mathbb{Q})$ is finite, then
$$\lim_{s \to 1} \dfrac{L(E,s)}{(s-1)^r \cdot \Omega^+_E \cdot \mathrm{Reg}_E} = \dfrac{\# \sha(E/\mathbb{Q}) \cdot \mathrm{Tam}(E)}{ (\# E(\mathbb{Q})_{\mathrm{tors}})^2}$$
\end{enumerate}
where 
\begin{itemize}
\item $\mathrm{Reg}_E = \mathrm{det} \left( ( \langle x_i , x_j \rangle_{\mathrm{NT}} )_{i,j} \right)$
with $E(\mathbb{Q})/E(\mathbb{Q})_{\mathrm{tors}} = \oplus^r_{i=1} \mathbb{Z} x_i$,
\item $\langle - , - \rangle_{\mathrm{NT}}$ is the N\'{e}ron--Tate height pairing,
\item $\mathrm{Tam}(E) = \prod_{\ell \vert N} [E(\mathbb{Q}_\ell) : E_0(\mathbb{Q}_\ell)]$
with the preimage of the non-singular points $E_0(\mathbb{Q}_\ell)$ under the mod $\ell$ reduction, and
\item $\Omega^+_E = \int_{E(\mathbb{R})} \vert \omega_E \vert$, $\omega_E$ is a $\mathbb{Z}$-basis of $ \Omega^1_{\mathcal{E}/\mathbb{Z}}$ with the N\'{e}ron model $\mathcal{E}$ of $E$.
\end{itemize}
\end{conj}
\subsection{Beilinson--Kato elements and the finiteness of rational points}
We first illustrate a simple application of (the bottom of) Beilinson--Kato elements to the arithmetic of elliptic curves.
\subsubsection{}
Let $p$ be a prime and $E$ an elliptic curve over $\mathbb{Q}$ without complex multiplication.
Let $T = \mathrm{Ta}_pE = \varprojlim_{n} E(\overline{\mathbb{Q}} )[p^k]$ be the $p$-adic Tate module
and
$V = V_pE = T \otimes_{\mathbb{Z}_p} \mathbb{Q}_p$ be the 2-dimensional $\mathbb{Q}_p$-vector space endowed with the continuous action of $G_{ \mathbb{Q} } = \mathrm{Gal}( \overline{\mathbb{Q}}/\mathbb{Q} )$.
Denote the corresponding Galois representation by $\rho : G_{ \mathbb{Q} } \to \mathrm{Aut}_{\mathbb{Q}_p}(V) \simeq \mathrm{GL}_2(\mathbb{Q}_p)$.
Let $\Sigma$ be a finite set of places of $\mathbb{Q}$ containing $p$, $\infty$, and bad reduction primes for $E$,
and denote by $\mathbb{Q}_{\Sigma}$ the maximal extension of $\mathbb{Q}$ unramified outside $\Sigma$.
Then the information of $E(\mathbb{Q})$ can be detected in Galois cohomology group $\mathrm{H}^1(\mathbb{Q}, V) = \mathrm{H}^1(\mathbb{Q}_{\Sigma}/\mathbb{Q}, V)$ via Kummer map
$$E(\mathbb{Q}) \otimes \mathbb{Q}_p \to \mathrm{H}^1(\mathbb{Q}, V)$$
which makes the connection between geometry and cohomology.
The same rule applies to the local case.
\subsubsection{} \label{subsubsec:local-galois-cohomology}
We first investigate the local nature of Galois cohomology at $p$.
\begin{exer}
Show that $\mathrm{H}^1(\mathbb{Q}_p, V)$ is a 2-dimensional $\mathbb{Q}_p$-vector space. (Hint: Use the local Euler characteristic formula.)
\end{exer}
The local Kummer map $E(\mathbb{Q}_p) \otimes \mathbb{Q}_p \hookrightarrow \mathrm{H}^1(\mathbb{Q}_p, V)$ embeds a 1-dimensional geometric object into a 2-dimensional cohomological one.
The Weil pairing
$$V \times V \to \mathbb{Q}_p(1)$$
induces a non-degenerate cup product pairing 
$$\langle -,- \rangle_p : \mathrm{H}^1(\mathbb{Q}_p, V) \times \mathrm{H}^1(\mathbb{Q}_p, V) \overset{\cup}{\to} \mathrm{H}^2(\mathbb{Q}_p, \mathbb{Q}_p(1)) \simeq \mathbb{Q}_p .$$
Under this pairing, we have the following orthogonality
$$E(\mathbb{Q}_p) \otimes \mathbb{Q}_p \perp E(\mathbb{Q}_p) \otimes \mathbb{Q}_p$$
due to local Tate duality.
We consider the commutative diagram
\[
\xymatrix@R=1.5em{
 \mathrm{H}^1(\mathbb{Q}_p, V) & \times &  \mathrm{H}^1(\mathbb{Q}_p, V) \ar[dd]^-{\mathrm{exp}^*_{\omega_E}} \ar[r] & \mathbb{Q}_p \ar@{=}[dd]\\
 E(\mathbb{Q}_p) \otimes \mathbb{Q}_p \ar@{^{(}->}[u] \\
 \mathbb{Q}_p \ar[u]^-{\simeq}_-{\mathrm{exp}_{\widehat{E}}} & \times & \mathbb{Q}_p \ar[r]  & \mathbb{Q}_p \\
 & & \mathrm{H}^0(E/\mathbb{Q}_p, \Omega^1) \ar[u]_-{\simeq}
}
\]
and explain the precise meaning of each term:
\begin{itemize}
\item  The map $\mathrm{exp}_{\widehat{E}} : \mathbb{Q}_p \to  E(\mathbb{Q}_p) \otimes \mathbb{Q}_p$ extends the formal exponential map
$\mathrm{exp}_{\widehat{E}} : p\mathbb{Z}_p \to \widehat{E}(p\mathbb{Z}_p)$.
Denote by $\omega^*_E$ the basis of the tangent space $\mathbb{Q}_p\omega^*_E$ of $E/\mathbb{Q}_p$ at the identity characterized by the natural pairing $\langle \omega_E , \omega^*_E \rangle = 1$.
If we identify the source $\mathbb{Q}_p$ with $\mathbb{Q}_p\omega^*_E$ by sending 1 to $\omega^*_E$, then the exponential map coincides with the Lie group exponential map \cite[pp. 64]{rubin-book}.
\item The latter $\mathbb{Q}_p$ in the diagram is isomorphic to the space of global 1-forms $\mathrm{H}^0(E/\mathbb{Q}_p, \Omega^1) = \mathbb{Q}_p \omega_E$, i.e. the cotangent space and  of $E/\mathbb{Q}_p$ at the identity,
by sending 1 to $\omega_E$.
\item The above dual exponential map $\mathrm{exp}^*_{\omega_E} : \mathrm{H}^1(\mathbb{Q}_p, V) \to \mathbb{Q}_p$ is the composition of Bloch--Kato's dual exponential map
$\mathrm{exp}^* : \mathrm{H}^1(\mathbb{Q}_p, V) \to \mathrm{H}^0(E/\mathbb{Q}_p, \Omega^1)$ and the above isomorphism $\mathrm{H}^0(E/\mathbb{Q}_p, \Omega^1) \simeq \mathbb{Q}_p$.
\item The bottom pairing of the diagram is given by multiplication: $(a, b) \mapsto a \cdot b$.
\end{itemize}
The characterization of the kernel of the dual exponential map is important for us.
We write $E(\mathbb{Q}_p) \otimes \mathbb{Q}_p := \left( E(\mathbb{Q}_p) \widehat{\otimes}_{\mathbb{Z}} \mathbb{Z}_p  \right) \otimes_{\mathbb{Z}_p} \mathbb{Q}_p$ for convenience.
\begin{equation} \label{eqn:kernel-dual-exp}
\mathrm{ker}(\mathrm{exp}^*_{\omega_E}) = E(\mathbb{Q}_p) \otimes \mathbb{Q}_p \subseteq \mathrm{H}^1(\mathbb{Q}_p, V) .
\end{equation} 
We now see the simplest form of Kato's work and feel its power for the first time.
\begin{thm}[Kato] \label{thm:kato-finiteness-mordell-weil}
There exists a global Galois cohomology class $z_{\mathbb{Q}} \in \mathrm{H}^1(\mathbb{Q}, V)$
such that
\[
\xymatrix@R=0em{
\mathrm{H}^1(\mathbb{Q}, V) \ar[r]^-{\mathrm{loc}_p} &
\mathrm{H}^1(\mathbb{Q}_p, V) \ar[r]^-{\mathrm{exp}^*} &
\mathbb{Q}_p \omega_E \\
z_{ \mathbb{Q} }  \ar@{|->}[rr] & & \mathrm{exp}^* (\mathrm{loc}_p ( z_{ \mathbb{Q} } ))
}
\]
and $$\mathrm{exp}^* (\mathrm{loc}_p ( z_{ \mathbb{Q} } )) = \dfrac{L^{(p)}(E,1)}{\Omega^+_E} \cdot \omega_E$$
where $L^{(p)}(E,1)$ is the $L$-value of $E$ at $s=1$ removing the Euler factor at $p$.
\end{thm}
\begin{cor}[Kato] \label{cor:kato-finiteness-mordell-weil}
If $\mathrm{rk}_{\mathbb{Z}}E(\mathbb{Q}) >0$, then $L(E, 1) = 0$.
\end{cor}
\begin{proof}
Let $P \in E(\mathbb{Q})$ be a point of infinite order.
Under the natural map
$$E(\mathbb{Q}) \hookrightarrow E(\mathbb{Q}_p) \to  E(\mathbb{Q}_p) \widehat{\otimes}_{\mathbb{Z}} \mathbb{Z}_p \to 
E(\mathbb{Q}_p) \otimes \mathbb{Q}_p ,$$
the image of $P$ generates $E(\mathbb{Q}_p) \otimes \mathbb{Q}_p$.
Since $z_{\mathbb{Q}} \in \mathrm{H}^1(\mathbb{Q}, V)$ and $P$ are global, the Artin reciprocity law implies that
$$\sum_{\ell \leq \infty} \langle \mathrm{loc}_\ell (z_\mathbb{Q}), P \rangle_{\ell} = 0.$$
Since $\mathrm{H}^1(\mathbb{Q}_\ell, V) = 0$ for every place $\ell \neq p$ (including the infinite place),
we have $\left\langle \mathrm{loc}_\ell (z_\mathbb{Q}), P \right\rangle_{p} = 0$.
By the self-orthogonality of $E(\mathbb{Q}_p) \otimes \mathbb{Q}_p$,
we have
$\mathrm{loc}_p(z_{\mathbb{Q}}) \in E(\mathbb{Q}_p) \otimes \mathbb{Q}_p$.
By (\ref{eqn:kernel-dual-exp}),
$\mathrm{exp}^* \circ \mathrm{loc}_p(z_{\mathbb{Q}})  = 0$.
Thus, $L(E,1) = 0$ by Theorem \ref{thm:kato-finiteness-mordell-weil}.
\end{proof}
The same result was also obtained by Kolyvagin via the use of Heegner points \cite{kolyvagin-euler-systems}.

This is the very starting point of Kato's Euler systems, and the cohomology class $z_{\mathbb{Q}}$ is just a part of a much deeper object.

\subsection{The finiteness of Selmer groups}
We recall the Selmer groups of $p$-power torsion points $E[p^\infty] = \bigcup_{n \geq 0} E[p^n]$ of $E$.
Then for $n \geq 1$, the \textbf{Selmer group of $E[p^n]$ over $\mathbb{Q}$} is defined by
$$\mathrm{Sel}(\mathbb{Q}, E[p^n]) := \mathrm{ker} \left(  \mathrm{H}^1(\mathbb{Q}_{\Sigma}/\mathbb{Q}, E[p^n])  \to \prod_{\ell \in \Sigma}  \dfrac{ \mathrm{H}^1(\mathbb{Q}_\ell, E[p^n]) }{ E( \mathbb{Q}_\ell )  \otimes \mathbb{Z}/p^n\mathbb{Z} }  \right)$$
and the \textbf{Selmer group of $E[p^\infty]$ over $\mathbb{Q}$} is defined by
$\mathrm{Sel}(\mathbb{Q}, E[p^\infty]) = \varinjlim_n \mathrm{Sel}(\mathbb{Q}, E[p^n])$ via the restriction maps.
\begin{exer}
Check that the definition of Selmer groups is independent of the choice of $\Sigma$. (The answer can be found in \cite[Cor. I.6.6]{milne-adt}.)
\end{exer}
The Selmer group $\mathrm{Sel}(\mathbb{Q}, E[p^\infty])$ fits into the fundamental exact sequence
\[
\begin{tikzcd}
0 \arrow[r] & E(\mathbb{Q}) \otimes \mathbb{Q}_p/\mathbb{Z}_p \arrow[r] & \mathrm{Sel}(\mathbb{Q}, E[p^\infty]) \arrow[r] & \sha(E/\mathbb{Q})[p^\infty] \arrow[r] & 0 ,
\end{tikzcd}
\]
and $\mathbb{Q}$ can be replaced by any algebraic extension of $\mathbb{Q}$.
\begin{rem}
The notion of Selmer groups can be generalized by imposing different local conditions (instead of the image of the local Kummer map).
This variation is very important for the Euler/Kolyvagin system argument.
\end{rem}
Kato's result on the finiteness of rational points (Corollary \ref{cor:kato-finiteness-mordell-weil}) can be strengthened including the finiteness of Tate--Shafarevich groups.
\begin{thm}[Kato] \label{thm:kato-finiteness-mordell-weil-sha}
Let $E$ be an elliptic curve without complex multiplication and $p$ be any prime.
\begin{enumerate}
\item If $L(E,1) \neq 0$, then $\mathrm{Sel}(\mathbb{Q}, E[p^\infty])$ is finite, so
$E(\mathbb{Q})$ and $\sha(E/\mathbb{Q})[p^\infty]$ are finite.
\item Let $F/\mathbb{Q}$ be a finite abelian extension and $\chi : \mathrm{Gal}(F/\mathbb{Q}) \to \mathbb{C}^\times$ be a character.
If the twisted $L$-value $L(E, \chi, 1) \neq 0$, then
the $\chi$-isotypic components of $E(F)$ and $\sha(E/F)[p^\infty]$ are finite, i.e.
\begin{align*}
E(F)^\chi & = E(F) \otimes_{\mathbb{Z}[\mathrm{Gal}(F/\mathbb{Q})], \chi} \mathbb{Z}[\chi] , \\
\sha(E/F)[p^\infty]^\chi & = \sha(E/F)[p^\infty] \otimes_{\mathbb{Z}[\mathrm{Gal}(F/\mathbb{Q})], \chi} \mathbb{Z}[\chi]
\end{align*}
are finite where $\mathbb{Z}[\chi]$ is the ring generated by the values of $\chi$ over $\mathbb{Z}$.
\end{enumerate}
\end{thm}
Theorem \ref{thm:kato-finiteness-mordell-weil-sha}.(1) was also obtained by Kolyvagin \cite{kolyvagin-euler-systems} and Kolyvagin's method is based on Heegner points.
For the CM case, see \cite{coates-wiles-bsd-1977,rubin-tate-shafarevich,rubin-main-conj-img-quad-fields}.

We discuss a refined version of Theorem \ref{thm:kato-finiteness-mordell-weil-sha}.(1) for all but finitely many primes in Theorem \ref{thm:kato-almost-divisibility}.
Theorem \ref{thm:kato-finiteness-mordell-weil-sha} certainly needs \emph{more than} one element $z_{\mathbb{Q}}$ and now the theory of Euler systems comes in more seriously.
\begin{rem}
Euler systems can be viewed as the cohomological shape of $L$-functions.
By definition, Euler systems consist of families of Galois cohomology classes $z_{\mathbb{Q}(\zeta_{mp^n})}$ over abelian extensions of $\mathbb{Q}$ with norm compatibility. This norm compatibility remember all the Euler factors of $L$-functions\footnote{This is where the name `Euler system' comes from.}.
Furthermore, Euler systems recover $L$-values via explicit reciprocity laws, which are very deep results in general.
On the other hand, Kolyvagin systems can be viewed as the cohomological shape of the Taylor expansion of $L$-functions.
\end{rem}

\subsection{Iwasawa theory for elliptic curves}
We briefly introduce Iwasawa theory for elliptic curves \emph{with good ordinary reduction} to some extent to illustrate the strength of Kato's Euler systems. A more detailed discussion will follow in the second lecture. See also \cite{mazur-IMC-for-E, greenberg-intro,greenberg-lnm}.

One of the mottoes of Iwasawa theory is the following counter-intuitive principle:
\begin{quotation}
``Infinite extensions are easier than finite extensions to deal with."
\end{quotation}
The Iwasawa main conjecture\footnote{John Coates first called the corresponding statement for totally real fields the ``main conjecture"  \cite{coates-durham, coates-iwasawa}.} for elliptic curves can be viewed as the cyclotomic deformation of the Birch and Swinnerton-Dyer conjecture. 
\subsubsection{}
Let $p$ be an odd prime and $E$ an elliptic curve with good ordinary reduction at $p$, i.e. $p \nmid N\cdot a_p(E)$
where $a_p(E) = p+1 - \# \widetilde{E}(\mathbb{F}_p) = a_p$ (by Theorem \ref{thm:modularity}).
Consider the tower of cyclotomic fields\\
$$
\xymatrix@R=1em{
\mathbb{Q}_\infty \ar@{-}[d] \ar@/_1pc/@{-}[dd]_-{\mathbb{Z}_p \simeq}   \\
\mathbb{Q}_n \ar@{-}[d]^-{\simeq \mathbb{Z}/p^n\mathbb{Z}} & \textrm{ and  } & 
{\txt{let $\Lambda = \mathbb{Z}_p\llbracket \mathrm{Gal}(\mathbb{Q}_\infty/\mathbb{Q}) \rrbracket = \varprojlim_{n} \mathbb{Z}_p[\mathrm{Gal}(\mathbb{Q}_n/\mathbb{Q})] 
 \simeq \mathbb{Z}_p\llbracket T \rrbracket$ \\   be the cyclotomic Iwasawa algebra, \\ which is isomorphic to the one-variable power series ring over $\mathbb{Z}_p$. }  } \\
\mathbb{Q} 
}
$$
Here, the isomorphism sends a topological generator $\gamma \in \mathrm{Gal}(\mathbb{Q}_\infty/\mathbb{Q})$ to $1+T$.
\subsubsection{}
We recall the notion of characteristic ideals in (one-variable) Iwasawa theory.
When $M$ be a finitely generated torsion $\Lambda$-module, $M$ is isomorphic to $\bigoplus_{i} \Lambda / (f_i)$ 
where $f_i \in \Lambda$ \emph{up to finite kernel and cokernel}.
 More details will be given in the second lecture ($\S$\ref{subsec:Sel-Q-infty-main-conjecture}).
Then the characteristic ideal of $M$ is defined by
$$\mathrm{char}_{\Lambda} (M) = \left( \prod_{i} f_i \right) \subseteq \Lambda .$$
The $p$-adic $L$-function
$L_p(E) \in \Lambda$ \emph{remembers} all twisted $L$-values $L(E, \chi , 1)$
by the interpolation formula
$$\chi(L_p(E)) \doteq \tau(\overline{\chi}) \cdot \dfrac{L(E, \chi, 1)}{\Omega^{\chi(-1)}_E}$$
for every finite order character
$\chi : \mathrm{Gal}(\mathbb{Q}_\infty/\mathbb{Q}) \to \mathbb{C}^\times$.
In order to describe the interpolation formula more precisely, we need to fix embeddings
$\overline{\mathbb{Q}}_p \hookleftarrow \overline{\mathbb{Q}} \hookrightarrow \mathbb{C}$.
\subsubsection{}
The following non-vanishing result is fundamental \cite{rohrlich-nonvanishing,rohrlich-nonvanishing-2}.
\begin{thm}[Rohrlich] \label{thm:rohrlich-non-vanishing}
The $p$-adic $L$-function $L_p(E)$ is non-zero.
\end{thm}
We now state the Iwasawa main conjecture  \`{a} la Greenberg--Mazur \cite{mazur-IMC-for-E,greenberg-lnm}.
\begin{conj}[Iwasawa main conjecture] \label{conj:imc-mazur-greenberg}
The Selmer group $\mathrm{Sel}(\mathbb{Q}_\infty, E[p^\infty])$ is $\Lambda$-cotorsion, and
\begin{equation} \label{eqn:imc-mazur-greenberg}
\left( L_p(E) \right) = \mathrm{char}_{\Lambda}(\mathrm{Sel}(\mathbb{Q}_\infty, E[p^\infty])^\vee) 
\end{equation}
as ideals of $\Lambda$.
\end{conj}
By using his Euler systems, Kato proved the one-sided divisibility of Conjecture \ref{conj:imc-mazur-greenberg} \cite[Thm. 17.4]{kato-euler-systems}.
\begin{thm}[Kato] \label{thm:kato-divisibility}
If the image of $\rho$ contains a conjugate of $\mathrm{SL}_2(\mathbb{Z}_p)$,
then $\mathrm{Sel}(\mathbb{Q}_\infty, E[p^\infty])$ is $\Lambda$-cotorsion and
\begin{equation} \label{eqn:kato-divisibility}
\left( L_p(E) \right) \subseteq \mathrm{char}_{\Lambda}(\mathrm{Sel}(\mathbb{Q}_\infty, E[p^\infty])^\vee) .
\end{equation}
\end{thm}
The opposite divisibility is proved by Skinner--Urban for a large class of elliptic curves \cite{skinner-urban} (Theorem \ref{thm:skinner-urban}).
Theorem \ref{thm:kato-divisibility} yields an upper bound of Selmer groups.
\subsubsection{}
\begin{thm} \label{thm:kato-almost-divisibility}
Let $E$ be an elliptic curve over $\mathbb{Q}$ without complex multiplication.
Suppose that $L(E,1) \neq 0$.
For all but finitely many primes $p$, we have divisibility
$$\#\sha(E/\mathbb{Q})[p^\infty] \Big\vert \dfrac{L(E,1)}{\Omega^+_E} .$$
\end{thm}
We explain how a proof of Theorem \ref{thm:kato-almost-divisibility} goes following \cite[Thm. 3.5.11 and Cor. 3.5.19]{rubin-book}.
By Serre's open image theorem \cite{serre-open-image}, we have the implication
\begin{quote}
$E$ has no CM $\Rightarrow$ $\rho_{E,p} : G_{\mathbb{Q}} \to \mathrm{Aut}_{\mathbb{Z}_p}(\mathrm{Ta}_pE)$ is surjective for $p \gg 0$.
\end{quote}
In addition, we remove the following finitely many primes:
\begin{enumerate}
\item all bad reduction primes,
\item all primes dividing Euler factors at $s=1$ at bad reduction primes,
\item 2,
\item all the primes dividing some constant $c_E$ (independent of $p$)
satisfying
$$``\mathrm{exp}^* \circ \mathrm{loc}_p z_{\mathbb{Q}} = c_E \cdot \dfrac{L^{(p)}(E,1)}{\Omega^+_E}",$$
and
\item all the primes dividing  some constant $M_E$ such that
if $E$ has good ordinary reduction at $p$ and $p$ divides $M_E$ then
$\mathrm{Sel}(\mathbb{Q}_{\infty}, E[p^\infty])^\vee$ has non-trivial finite $\Lambda$-submodule.
See \cite[Cor. 3.5.19]{rubin-book} for details.
For example, if $p$ does not divide $\prod_{q \vert N} \#E(\mathbb{Q}_q)_{\mathrm{tors}}$, then 
$\mathrm{Sel}(\mathbb{Q}_{\infty}, E[p^\infty])^\vee$ has no non-trivial finite $\Lambda$-submodule \cite[Cor. 3.5.18]{rubin-book}.
\end{enumerate}
When $p$ does not divide $\# E(\mathbb{F}_p)$, the standard Euler system argument over $\mathbb{Q}$ is enough.
When $p$ divides $\# E(\mathbb{F}_p)$, we need to work over $\mathbb{Q}_\infty$.
By using Theorem \ref{thm:kato-divisibility} and the non-existence of non-trivial finite $\Lambda$-submodules,
Kato's divisibility (\ref{eqn:kato-divisibility})
implies that
$$\# \left\vert \mathrm{Sel}(\mathbb{Q}_{\infty}, E[p^\infty])^{\mathrm{Gal}(\mathbb{Q}_{\infty}/\mathbb{Q})} \right\vert \textrm{ divides } \textbf{1}(L_p(E)).$$
The conclusion follows from the interpolation formula of the $p$-adic $L$-function (\ref{eqn:interpolation-trivial-character}) and Mazur's control theorem (Proposition \ref{prop:control-thoerem}).
\subsubsection{}
We have the following application.
\begin{cor} Let $E$ be an elliptic curve over $\mathbb{Q}$ without complex multiplication.
If $L(E, 1) \neq 0$, then $E(\mathbb{Q})$ and $\sha(E/\mathbb{Q})$ are finite.
\end{cor}
\begin{proof}
By Theorem \ref{thm:kato-finiteness-mordell-weil-sha}.(1), $E(\mathbb{Q})$ is finite and $\sha(E/\mathbb{Q})[p^\infty]$ is finite for every prime $p$.
By Theorem \ref{thm:kato-almost-divisibility}.(1), $\sha(E/\mathbb{Q})[p^\infty]$ is trivial for all but finitely many prime $p$.
\end{proof}
From Theorem \ref{thm:kato-finiteness-mordell-weil-sha}.(2) and Theorem \ref{thm:rohrlich-non-vanishing}, we also deduce the following corollary.
\begin{cor} \label{cor:finitely-generated-Q-infty}
Let $E$ be an elliptic curve over $\mathbb{Q}$ without complex multiplication.
Then $E(\mathbb{Q}_{\infty})$ is (still!) finitely generated.
\end{cor}
We can understand Corollary \ref{cor:finitely-generated-Q-infty} as a behavior of the ``function $E$"; in other words, the assignment 
$$\mathbb{Q}_n \mapsto E(\mathbb{Q}_n)$$
does not grow too fast. This growth behavior shows the philosophy of Iwasawa theory very well.

\subsubsection{} \label{subsubsec:iwasawa-cohomology}
How does Kato's Euler system play the role in Theorem \ref{thm:kato-divisibility}?
Let $$z_{\mathbb{Q}_{\infty}} = \varprojlim_{n} z_{\mathbb{Q}_{n}} \in \mathrm{H}^1_{\mathrm{Iw}}(\mathbb{Q}, T) = \varprojlim_{n}
\mathrm{H}^1(\mathbb{Q}_n, T)$$
be Beilinson--Kato's zeta element over $\mathbb{Q}_{\infty}$
where the inverse limit is taken with respect to the corestriction, and the norm compatibility $\varprojlim_{n} z_{\mathbb{Q}_{n}}$ is a part of the Euler system relation.
The punchline is that we can realize the $p$-adic $L$-function $L_p(E)$ by using $z_{\mathbb{Q}_{\infty}}$.
More precisely, there exists a Coleman map
$$\mathrm{Col} : \mathrm{H}^1_{\mathrm{Iw}}(\mathbb{Q}_p, T) \to \Lambda$$
sending $\mathrm{loc}_p ( z_{\mathbb{Q}_{\infty}} )$ to $L_p(E)$.
If $E$ has good supersingular reduction, the target of the Coleman map becomes larger than $\Lambda$, but the signed Coleman maps work as in the ordinary case.
By using the Coleman map, we have the equivalence between
\begin{itemize}
\item Kato's main conjecture for elliptic curves (Conjecture \ref{conj:kato-main-conj-sel0}), and
\item Mazur--Greenberg's main conjecture for elliptic curves with good ordinary reduction (Conjecture \ref{conj:imc-mazur-greenberg}).
\end{itemize}
There is no assumption on the reduction type in Kato's formulation. We discuss this equivalence in the second lecture ($\S$\ref{subsec:global-duality}).

The notion of Coleman map is purely local and the generalizations of the Coleman map is an important part of $p$-adic Hodge theory.
See \cite{perrin-riou-local-iwasawa, nakamura-local-iwasawa} for example.

\subsection{Refining the Euler system argument}
In the Euler system argument \cite{rubin-book}, the \emph{derivative process} plays the crucial role to produce the \emph{ramified} classes over $\mathbb{Q}$ from the unramified cohomology classes over ramified abelian extensions of $\mathbb{Q}$. We can bound Selmer groups with these ramified classes via the global Poitou--Tate duality.
A more careful organization of the derivative process leads to Mazur--Rubin's theory of Kolyvagin systems \cite{mazur-rubin-book}, which has the following advantages:
\begin{enumerate}
\item the theory of Kolyvagin systems provides a mod $p$ numerical criterion to have equality, not divisibility. See \cite{kks}.
\item the theory of Kolyvagin systems provides the structure theorem of Selmer groups. See also \cite{kolyvagin-selmer}.
\end{enumerate}
In Mazur's diagram below
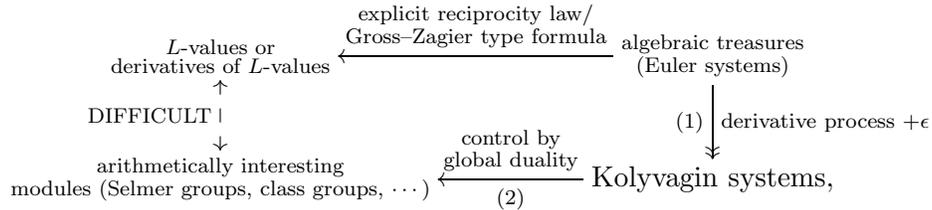
\[
\xymatrix{
{\substack{\textrm{$L$-values or} \\ \textrm{derivatives of $L$-values} } } \ar@{<-->}[d]_-{\textrm{DIFFICULT}} & & {\substack{\textrm{algebraic treasures} \\ \textrm{(Euler systems)} } } \ar[ll]_-{ {\substack{\textrm{explicit reciprocity law/} \\ \textrm{Gross--Zagier type formula} } } } \ar@{->>}[d]^-{ \textrm{derivative process $+ \epsilon$} }_-{(1)} \\
{\substack{\textrm{arithmetically interesting} \\ \textrm{modules (Selmer groups, class groups, $\cdots$)} } } & & \textrm{Kolyvagin systems} , \ar[ll]_-{ {\substack{\textrm{control by} \\ \textrm{global duality} } } }^-{(2)}
}
\]
the mod $p$ criterion (1) concerns the surjectivity of the map from Euler systems to Kolyvagin systems and the structure theorem (2) is a refined version of the global duality argument.
\subsection{Where do Euler systems come from?}
One big question is: 
\begin{ques}
Where do Euler systems come from?
\end{ques}
This seems the most non-trivial question in the theory of Euler systems.
There was a handful of Euler systems a while ago (e.g. \cite{bertolini-castella-darmon-dasgupta-prasanna-rotger}), but 
a general strategy to construct Euler systems has been developed thanks to the series of the important works of Loeffler--Zerbes and their collaborators. See \cite{loeffer-zerbes-icm} for the progress of Loeffler--Zerbes' program on the construction of Euler systems for automorphic Galois representations.
Also, M. Sangiovanni and C. Skinner have recently announced another general strategy to construct Euler systems.

\subsection{A very$^2$ rough picture} \label{subsec:rough-picture}
In order to describe 
$p$-adic zeta element ${}_{c,d} z^{(p)}_{m}(f, r, r' , \xi, S)$
and zeta modular form ${}_{c,d} z_{m}(f, r, r' , \xi, S)$, we draw the following very rough picture. We hope it provides a guideline to read \cite{kato-euler-systems}:
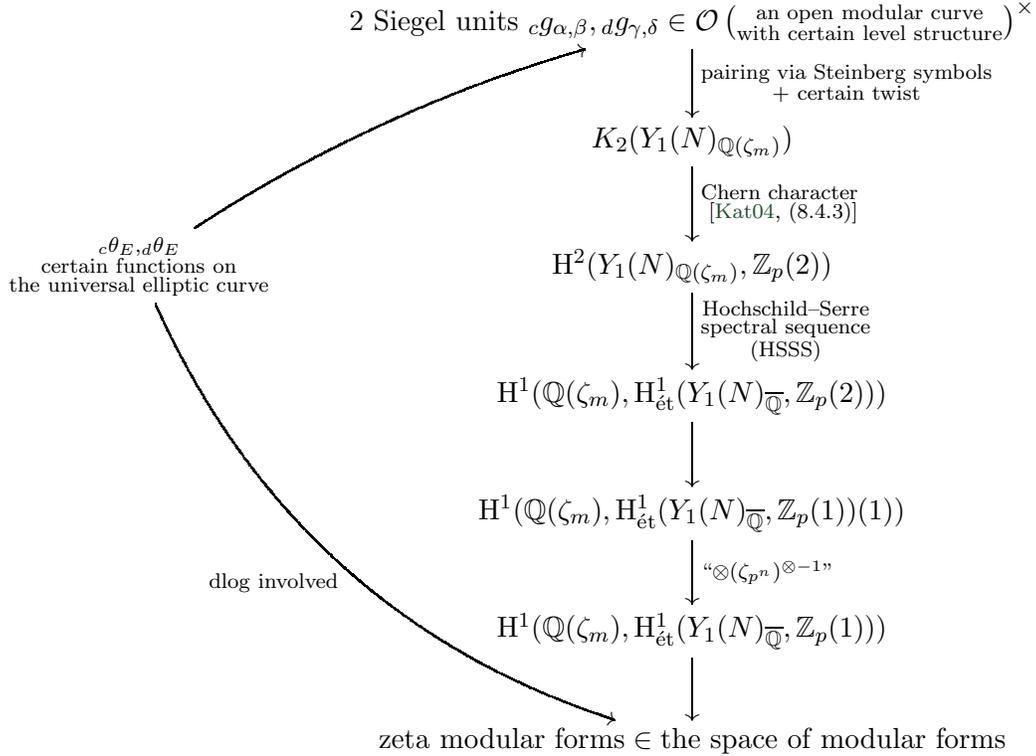
\[
\xymatrix{
  & \textrm{2 Siegel units } {}_c g_{\alpha, \beta}, {}_d g_{\gamma, \delta} \in \mathcal{O}\left(  {\substack{\textrm{an open modular curve} \\ \textrm{with certain level structure} } } \right)^\times \ar[d]^-{ {\substack{\textrm{pairing via Steinberg symbols} \\ \textrm{$+$ certain twist} } } }\\
  &  K_2(Y_1(N)_{\mathbb{Q}(\zeta_m)}) \ar[d]^-{ {\substack{\textrm{Chern character} \\ \textrm{\cite[(8.4.3)]{kato-euler-systems}} } } } \\
{\substack{ {}_c \theta_E, {}_d \theta_E  \\ \textrm{certain functions on} \\ \textrm{the universal elliptic curve} } } \ar@/^1pc/[uur] \ar@/_3pc/[ddddr]_-{\textrm{$\mathrm{dlog}$ involved}}
 &  \mathrm{H}^2(Y_1(N)_{\mathbb{Q}(\zeta_m)}, \mathbb{Z}_p(2)) \ar[d]^-{ {\substack{\textrm{Hochschild--Serre} \\ \textrm{spectral sequence} \\ \textrm{(HSSS)} } } } \\
  &  \mathrm{H}^1(\mathbb{Q}(\zeta_m), \mathrm{H}^1_{\et} (Y_1(N)_{\overline{\mathbb{Q}}}, \mathbb{Z}_p(2))) \ar[d] \\
  &  \mathrm{H}^1(\mathbb{Q}(\zeta_m), \mathrm{H}^1_{\et} (Y_1(N)_{\overline{\mathbb{Q}}}, \mathbb{Z}_p(1))(1)) \ar@{-->}[d]^-{ { \textrm{``$\otimes (\zeta_{p^n})^{\otimes -1}$"} } } \\
  &  \mathrm{H}^1(\mathbb{Q}(\zeta_m), \mathrm{H}^1_{\et} (Y_1(N)_{\overline{\mathbb{Q}}}, \mathbb{Z}_p(1))) \ar[d]^-{ {\substack{\textrm{dual exponential} \\ \textrm{\cite[Thm. 9.6]{kato-euler-systems}} } } } \\
  &  \textrm{zeta modular forms} \in \textrm{the space of modular forms}
}
\]
We will elaborate this picture more in the last lecture and here are quick remarks:
\begin{itemize}
\item $\mathcal{O}\left(  {\substack{\textrm{an open modular curve} \\ \textrm{with certain level structure} } } \right)^\times$ means the ring of invertible functions on an open modular curve with certain level structure as described in $\S$\ref{subsubsec:more-modular-curves}. 
\item The relation between $({}_c\theta_E, {}_d\theta_E)$ and $(m, N)$ can be given explicitly and it is more or less hidden in the word ``$+$ certain twist". See $\S$\ref{subsec:changing-modular-curves} for the hint.
\item In order to take the Soul\'{e} twist ``$\otimes (\zeta_{p^n})^{\otimes -1}$" properly, we need to work over $\mathbb{Q}(\zeta_{mp^\infty})$ and with torsion coefficients first and go down to $\mathbb{Q}(\zeta_{m})$. See \cite[(8.4.3)]{kato-euler-systems}.
\end{itemize} 
The zeta modular form can be viewed as a product of two weight 1 Eisenstein series.
The projection to an elliptic curve $E$ yields the $L$-value formula for $E$ via the Rankin--Selberg method.
When we project the zeta modular form to $E$, we need to invert $p$, so the integrality issue becomes subtle in general \cite[$\S$6.3]{kato-euler-systems}. 
When we use the Chern character map, the parameters $(r, r')$ in ${}_{c,d} z_{m}(f, r, r' , \xi, S)$ are determined. (Since we deal with elliptic curves only here, $f$ is a weight 2 modular form.)

In this picture, ${}_{c,d} z^{(p)}_{m}(f, r, r' , \xi, S)$ lies in the image of $p$-adic zeta element 
under the specialization map to the eigenform $f$
$$\mathrm{H}^1(\mathbb{Q}(\zeta_m), \mathrm{H}^1_{\et} (Y_1(N)_{\overline{\mathbb{Q}}}, \mathbb{Z}_p(1))) \to \mathrm{H}^1(\mathbb{Q}(\zeta_m), V_f(1)) = \mathrm{H}^1(\mathbb{Q}(\zeta_m), V_p(E)) ,$$
and ${}_{c,d} z_{m}(f, r, r' , \xi, S)$ lies in the image of the zeta modular form
under the similar specialization map
$$S_2(\Gamma_0(N)) \otimes \mathbb{Q}(\zeta_m) \to S(f) \otimes \mathbb{Q}(\zeta_m)  \otimes \mathbb{Q}_p$$
where $S(f)$ is a one-dimensional quotient space of $S_2(\Gamma_0(N))$ generated by the Hecke eigensystem of $f$.
Although the target spaces are rational, the actual images of the above elements are integral.

\section{Lecture 2: Cyclotomic Iwasawa theory for elliptic curves}
Let $\mathbb{Q}_\infty$ be the cyclotomic $\mathbb{Z}_p$-extension of $\mathbb{Q}$.
The following picture shows how Kato's Euler systems,  Kato's Kolyvagin systems,  and the Iwasawa main conjecture are related. 
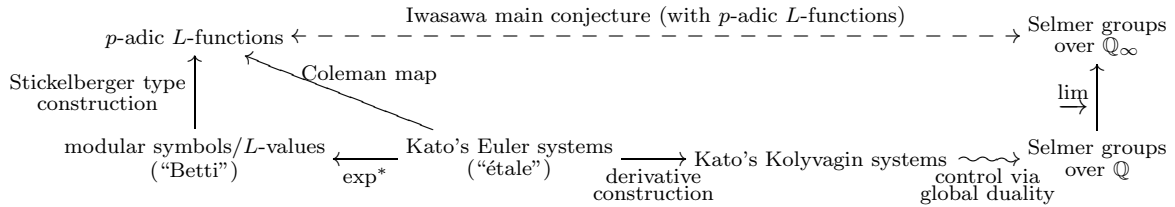
\[
\xymatrix{
\scriptsize{ \textrm{$p$-adic $L$-functions} } \ar@{<-->}[rrr]^-{\textrm{Iwasawa main conjecture (with $p$-adic $L$-functions)}} & & & 
{\substack{ \textrm{Selmer groups} \\ \textrm{over } \mathbb{Q}_\infty}  }
 \\
{\substack{ \textrm{modular symbols/$L$-values} \\ \textrm{(``Betti")} }  } \ar[u]^-{ {\substack{ \textrm{Stickelberger type} \\ \textrm{construction} }  } } &
 {\substack{ \textrm{Kato's Euler systems} \\ \textrm{(``\'{e}tale")} }  } \ar[l]^-{\mathrm{exp}^*} \ar[r]_-{ {\substack{ \textrm{derivative} \\ \textrm{construction} }  } } \ar[lu]_-{\textrm{Coleman map}} 
& \scriptsize{  \textrm{Kato's Kolyvagin systems} } \ar@{~>}[r]_-{ {\substack{ \textrm{control via} \\ \textrm{global duality} }  } } & {\substack{ \textrm{Selmer groups} \\ \textrm{over } \mathbb{Q} }  } \ar[u]^-{\substack{ \varinjlim \\ \textrm{over } \mathbb{Q}_n}  }
}
\]

The goal of the second lecture is to explain some part of the above picture in detail.
One may ask that the ``de Rham" aspect is missing in this picture. 
This nature is more or less hidden in the dual exponential map since the Rankin--Selberg method is used in the proof of Kato's zeta value formula \cite[$\S$7]{kato-euler-systems}. See also \S\ref{subsec:coleman-maps}.
\subsection{$p$-adic $L$-functions and modular symbols: Betti construction} \label{subsec:betti-construction}
We review the modular symbol construction of $p$-adic $L$-functions following \cite{mtt,pollack-oms,bellaiche-book}.
\subsubsection{$L$-values}
How to extract $L$-values from (modular) elliptic curves?

Let $f_E(z) = \sum_{n \geq 1} a_n \cdot q^n$ where $q = e^{2 \pi i z}$.
We can express the $L$-values in terms of integrals \emph{very simply  if we ignore the convergence issue} \cite{pollack-oms}.
\begin{align*}
2 \pi i \cdot \int^{0}_{i\infty} f_E(z) dz
& =  2 \pi i \cdot \int^{0}_{i\infty} \sum_{n \geq 1} a_n \cdot e^{2 \pi i n z} dz \\
& = 2 \pi i \cdot \sum_{n \geq 1} a_n \cdot \int^{0}_{i\infty}  e^{2 \pi i n z} dz \\
& = \sum_{n \geq 1} a_n  \cdot n^{-1} \cdot \left( \left.  e^{2 \pi i n z} \right\vert^{0}_{i \infty} \right)\\
& = \sum_{n \geq 1} a_n  \cdot n^{-1} \\
& = L(E,1) .
\end{align*}
\begin{exer}
Where did we ignore the convergence issue?
\end{exer}
\subsubsection{Twists}
We closely follow \cite[$\S$8]{mtt} (and \cite[$\S$5.4.1]{bellaiche-book}).

Let $\chi : (\mathbb{Z}/d\mathbb{Z})^\times \to \mathbb{C}^\times$ be a Dirichlet character of conductor $d$.
Then how can we express the twisted $L$-value
$L(E, \chi, 1)$ in terms of integrals?
Consider
$$\tau(n , \chi) = \sum_{a \in (\mathbb{Z}/d\mathbb{Z})^\times} \chi(a) \cdot e^{2 \pi i n \cdot (a/d)}$$
where the sum should be viewed as the sum over $a = 1, \cdots, d$ with $(a,d) = 1$.
Then
$\tau(\chi) = \tau(1, \chi)$ is the Gauss sum
and $\tau(n, \chi) = \overline{\chi(n)} \cdot \tau(\chi)$.
Consider the twisted modular form
\begin{align*}
f_{\overline{\chi}}(z) & = \sum_{n \geq 1} \overline{\chi}(n) \cdot a_n \cdot e^{2 \pi i n z} \\
& =  \dfrac{1}{\tau(\chi)} \cdot \sum_{n \geq 1} \tau(n, \chi) \cdot a_n \cdot e^{2 \pi i n z} \\
& =  \dfrac{1}{\tau(\chi)} \cdot \sum_{n \geq 1} \sum_{a \in (\mathbb{Z}/n\mathbb{Z})^\times} \chi(a) \cdot a_n \cdot e^{2 \pi i n (z + a/d)} .
\end{align*}
By Birch's lemma \cite[(8.3)]{mtt}, we have
\begin{align*}
L(E, \overline{\chi}, 1) & = \dfrac{2 \pi i}{\tau(\chi)} \cdot \sum_{a \in (\mathbb{Z}/n\mathbb{Z})^\times} \chi(a) \cdot \int^{a/d}_{i \infty} f(z) dz \\
& = \dfrac{\tau(\overline{\chi})}{d} \cdot \sum_{a \in (\mathbb{Z}/n\mathbb{Z})^\times} \chi(a) \cdot 2 \pi i \int^{-a/d}_{i \infty} f(z) dz
\end{align*}
These computations lead to the following concept.
Let $X_0(N)(\mathbb{C}) = \Gamma_0(N) \backslash ( \mathfrak{h} \cup \mathbb{P}^1(\mathbb{Q}) ) $ be the complex points of compactified modular curve $X_0(N)$.
For a pair $\alpha , \beta \in  \mathbb{P}^1(\mathbb{Q}),$ we define
$$\left\lbrace \alpha , \beta \right\rbrace_f := 2\pi i \cdot \int^{\beta}_{\alpha} f(z) dz$$
and it is called a (non-normalized) modular symbol.
\subsubsection{Modular elements of Mazur--Tate}
Fix an isomorphism
$$(\mathbb{Z}/d\mathbb{Z})^\times \simeq \mathrm{Gal}(\mathbb{Q}(\zeta_d)/\mathbb{Q})$$
by sending $a$ to $(\sigma_a : \zeta_d \mapsto \zeta^a_d)$.
Let $\chi : (\mathbb{Z}/d\mathbb{Z})^\times \to \mathbb{C}^\times$ be a Dirichlet character and we extend it linearly
$$\chi :\mathbb{C}[\mathrm{Gal}(\mathbb{Q}(\zeta_d)/\mathbb{Q})] \to \mathbb{C}$$
by sending $\sigma \in \mathrm{Gal}(\mathbb{Q}(\zeta_d)/\mathbb{Q})$ to $\chi(\sigma)$.
Under this map, we have
$$\sum_{a \in (\mathbb{Z}/d\mathbb{Z})^\times} \left\lbrace  i\infty , a/d \right\rbrace_f \cdot \sigma_a \mapsto \tau(\chi) \cdot L(E, \overline{\chi}, 1) .$$
In other words, (non-normalized) Mazur--Tate elements know twisted $L$-values.

\subsubsection{Some hints on $p$-adic $L$-functions}
Let $\chi : (\mathbb{Z}/p^n\mathbb{Z})^\times \to \mathbb{C}^\times$ be a Dirichlet character of modulus $p^n$.
Then  we have
$$L(E, \overline{\chi}, 1) = \dfrac{2 \pi i}{\tau(\chi)} \cdot \sum_{a \in (\mathbb{Z}/p^n\mathbb{Z})^\times} \chi(a) \cdot \int^{a/p^n}_{i \infty}  f(z) dz$$
and the summation looks like a Riemann sum since
$$\mathbb{Z}^\times_p = \coprod_{a \in (\mathbb{Z}/p^n\mathbb{Z})^\times} (a+p^n \mathbb{Z}_p) ,$$
if $\chi$ is a locally constant function on $\mathbb{Z}^\times_p$, then
the twisted $L$-value can be understood as a $p$-adic integration.
\begin{rem}
Where does the distribution relation come from? Hecke operators!
\end{rem}
\subsubsection{Normalizations}
Write 
$\lbrace r \rbrace_f = \lbrace i \infty,  r \rbrace_f$ for $r \in \mathbb{P}^1(\mathbb{Q})$
and then we symmetrize integrals:
$$\lbrace r \rbrace^{\pm}_f = 2 \pi i \cdot \left(  \int^r_{i \infty} f(z)dz \pm \int^{-r}_{i \infty} f(z)dz \right) $$
We recall the algebraicity result of Manin and Shimura.
\begin{thm}[Manin, Shimura]
There exist a pair of complex periods $\Omega^{\pm}_f \in \mathbb{C}^\times$ such that
$$\lbrace r \rbrace^{\pm}_f \in \mathbb{Q}_f \cdot \Omega^{\pm}_f$$
where $\mathbb{Q}_f = \mathbb{Q}(a_n(f); n)$ is the number field generated by Fourier coefficients of $f$.
\end{thm}
Write
$$[r]^{\pm}_f = \lbrace r \rbrace^{\pm}_f / \Omega^{\pm}_f \in \mathbb{Q}_f \subseteq \overline{\mathbb{Q}}$$
and 
$$[r]_f = [r]^{+}_f  + [r]^{-}_f .$$

\subsubsection{Hecke action and distribution relations}
For $\alpha , \beta \in \mathbb{P}^1(\mathbb{Q})$, 
consider the homology class
$\lbrace \alpha , \beta \rbrace \in \mathrm{H}_1(X_0(N), \mathbb{Q})$.
Here, having the $\mathbb{Q}$-coefficients in the homology is due to Manin \cite{manin-parabolic}.
For a prime $\ell$ not dividing $N$, we have
$$T_\ell \lbrace \alpha , \beta \rbrace = \lbrace \ell \alpha , \ell \beta \rbrace +
\sum^{\ell-1}_{r =0}\lbrace \dfrac{\alpha +r}{\ell}  , \dfrac{\beta +r}{\ell} \rbrace .$$
When $\ell = p$, we want to remove $\lbrace p \alpha , p \beta \rbrace$ in the above sum in order to have the correct Riemann sum relation.
This is why we need the $p$-stabilization process.

Let $$\theta_{\mathbb{Q}(\zeta_M)} = \sum_{a \in (\mathbb{Z}/M\mathbb{Z})^\times} [a/M]_f \cdot \sigma_a \in \mathbb{Q}[\mathrm{Gal}( \mathbb{Q}(\zeta_M)/\mathbb{Q} )]$$
 be the modular element, which is the Stickelberger element for elliptic curves.
Varying $M$, we have two natural maps among them.
Let
$$ \pi : \mathbb{Q}[\mathrm{Gal}( \mathbb{Q}(\zeta_{M\ell})/\mathbb{Q} ) ] \to \mathbb{Q}[\mathrm{Gal}( \mathbb{Q}(\zeta_{M})/\mathbb{Q} ) ]$$
be the natural projection
and 
$$\nu  : \mathbb{Q}[\mathrm{Gal}( \mathbb{Q}(\zeta_{M})/\mathbb{Q} ) ] \to \mathbb{Q}[\mathrm{Gal}( \mathbb{Q}(\zeta_{M\ell})/\mathbb{Q} ) ]$$ be the norm map sending $\sigma \in \mathbb{Q}(\zeta_{M})/\mathbb{Q} )$ to the sum of liftings $\sum_{\pi(\tau) = \sigma} \tau$.
\begin{prop} \label{prop:compatiblity-modular-elements}
$$\pi \left( \theta_{\mathbb{Q}(\zeta_M)} \right)
=
\left\lbrace
\begin{array}{ll}
 a_\ell \cdot \theta_{\mathbb{Q}(\zeta_M)} - \nu \left(  \theta_{\mathbb{Q}(\zeta_{M/\ell})} \right) & \textrm{ if } \ell \vert M , \\
 (a_\ell - \sigma_\ell - \sigma^{-1}_{\ell}) \cdot \theta_{\mathbb{Q}(\zeta_M)}  & \textrm{ if } \ell \not\vert M .
\end{array} \right.$$
\end{prop}
\begin{proof}
We left it to the reader as an exercise. (Hint: use Hecke action.)
\end{proof}
In favorable situations, the modular elements are integral with N\'{e}ron periods.
\begin{rem} \label{rem:integrality}
If $E[p]$ is irreducible and the Manin constant is prime to $p$, then
$$\theta_{\mathbb{Q}(\zeta_{p^n})}
\in \mathbb{Z}_p[\mathrm{Gal}( \mathbb{Q}(\zeta_{p^n})/\mathbb{Q} ) ] $$
for every $n \geq 1$. 
See \cite[pp. 200--201]{kurihara-invent} for details.
\end{rem}
\subsubsection{$p$-adic $L$-functions}
We modify the three term relation (the first one in Proposition \ref{prop:compatiblity-modular-elements} with $\ell = p$) of modular elements
in order to have the norm compatible sequence. This is where the $p$-stabilization plays the role.
Take $\alpha$ a root of $X^2 - a_p X + p =0$
and put
$$\theta^{\alpha}_{\mathbb{Q}(\zeta_{p^n})} = \dfrac{1}{\alpha^n} \cdot \left( \theta_{\mathbb{Q}(\zeta_{p^n})} 
- \dfrac{1}{\alpha} \cdot \theta_{\mathbb{Q}(\zeta_{p^{n-1}})} \right) .$$
Then $\left\lbrace 
\theta^{\alpha}_{\mathbb{Q}(\zeta_{p^n})}
\right\rbrace$
forms a projective system, i.e. for $n \geq 1$, we have
$$\pi\left( \theta^{\alpha}_{\mathbb{Q}(\zeta_{p^{n+1}})} \right) = \theta^{\alpha}_{\mathbb{Q}(\zeta_{p^n})} .$$
Suppose that $\alpha$ is a $p$-adic unit (i.e. $E$ has good ordinary reduction at $p$).
Then the $p$-adic $L$-function $L_p(E)$ of $E$ is defined by
\[
\xymatrix@R=0em{
\varprojlim_n \mathbb{Z}_p\llbracket \mathrm{Gal}( \mathbb{Q}(\zeta_{p^n})/\mathbb{Q}  ) \rrbracket \ar@{=}[r] &  \mathbb{Z}_p\llbracket \mathrm{Gal}( \mathbb{Q}(\zeta_{p^\infty})/\mathbb{Q}  ) \rrbracket \ar@{->>}[r] & \mathbb{Z}_p\llbracket \mathrm{Gal}( \mathbb{Q}_\infty/\mathbb{Q}  )  \rrbracket \\
\varprojlim_n \theta^{\alpha}_{\mathbb{Q}(\zeta_{p^n})} \ar@{|->}[r] & \theta^{\alpha}_{\mathbb{Q}(\zeta_{p^\infty})} \ar@{|->}[r] & L_p(E) .
}
\]
We have the following interpolation formulas:
\begin{align} \label{eqn:interpolation-trivial-character}
\begin{split}
\mathbf{1}\left( \theta^{\alpha}_{\mathbb{Q}(\zeta_{p^\infty})} \right) & = 
\left(1 - \dfrac{1}{\alpha} \right)^2 \cdot \theta_{\mathbb{Q}} \\
& = \left(1 - \dfrac{1}{\alpha} \right)^2 \cdot \dfrac{L(E,1)}{\Omega^+_E} .
\end{split}
\end{align}
If $\chi$ is a Dirichlet character of conductor $p^n$ ($\neq$ 1),
then
$$\chi\left( \theta^{\alpha}_{\mathbb{Q}(\zeta_{p^\infty})} \right) = \dfrac{\tau(\chi)}{\alpha^n} \cdot \dfrac{L(E, \overline{\chi},1)}{\Omega^{\chi(-1)}_E}.$$
These interpolation formulas also characterize $\theta^{\alpha}_{\mathbb{Q}(\zeta_{p^\infty})}$\footnote{This is related to being of the non-critical slope.}.

\subsection{Selmer groups over $\mathbb{Q}_{\infty}$ and the main conjecture} \label{subsec:Sel-Q-infty-main-conjecture}
\subsubsection{}
Let $M$ be any finitely generated torsion $\Lambda$-module.
Then there exists a $\Lambda$-homomorphism
$$M \to \Lambda / f^{e_1}_1 \oplus \cdots \oplus  \Lambda / f^{e_r}_r$$
with finite kernel and cokernel. The map between finitely generated $\Lambda$-modules with finite kernel and cokernel is called a \textbf{pseudo-isomorphism}.
Then 
$(f^{e_1}_1), \cdots ,( f^{e_r}_r)$ are uniquely determined by $M$, and
the characteristic ideal of $M$ is defined by
$$\mathrm{char}_{\Lambda}(M) = \left(  \prod^{r}_{i=1} f^{e_i}_i \right) \subseteq \Lambda .$$
This notion of characteristic ideals is an analogue of characteristic polynomials for linear algebra over $\Lambda$.
See \cite[$\S$13.2]{washington} for the basic of linear algebra over $\Lambda$.
\begin{thm}[Kato]
Assume that the image of $\rho$ contains a conjugate of $\mathrm{SL}_2(\mathbb{Z}_p)$ and $E$ has good ordinary reduction at $p$.
Then
$\mathrm{Sel}(\mathbb{Q}_\infty, E[p^\infty])^\vee = 
\mathrm{Hom}_{\mathbb{Z}_p}(\mathrm{Sel}(\mathbb{Q}_\infty, E[p^\infty]), \mathbb{Q}_p/\mathbb{Z}_p)$
is a finitely generated torsion $\Lambda$-module.
\end{thm}
\begin{proof}
This follows from Theorem \ref{thm:rohrlich-non-vanishing} and Theorem \ref{thm:kato-divisibility}.
See also \cite[Thms. 12.4 and 17.4]{kato-euler-systems} for the statement without the large image assumption.
\end{proof}
We recall Mazur's control theorem \cite{mazur-IMC-for-E, greenberg-lnm}.
\begin{prop}[Mazur] \label{prop:control-thoerem}
Assume $E$ has good ordinary reduction at $p$.
Then the natural restriction map
$$\mathrm{Sel}(\mathbb{Q}, E[p^\infty]) \to \mathrm{Sel}(\mathbb{Q}_{\infty}, E[p^\infty])^{\Gamma}$$
has finite kernel and cokernel
where $\Gamma = \mathrm{Gal}(\mathbb{Q}_{\infty}/\mathbb{Q})$.
\end{prop}

\subsubsection{}
We recall Conjecture \ref{conj:imc-mazur-greenberg} and Theorem \ref{thm:kato-divisibility}.
\begin{conj}[Mazur, Greenberg]
Assume $E$ has good ordinary reduction at $p$.
Then
$$( L_p(E) ) = \mathrm{char}_{\Lambda}( \mathrm{Sel}(\mathbb{Q}_\infty, E[p^\infty])^\vee )$$
as ideals of $\Lambda$.
\end{conj}
\begin{thm}[Kato]
Assume $E$ has good ordinary reduction at $p$.
If the image of $\rho$ contains a conjugate of $\mathrm{SL}_2(\mathbb{Z}_p)$,
then
$$( L_p(E) ) \subseteq \mathrm{char}_{\Lambda}( \mathrm{Sel}(\mathbb{Q}_\infty, E[p^\infty])^\vee )$$
as ideals of $\Lambda$.
Without the image assumption, we still have the same divisibility as ideals of $\Lambda \otimes \mathbb{Q}_p$.
\end{thm}
By using a completely different idea (Eisenstein congruences), Skinner--Urban proved the opposite divisibility \cite{skinner-urban}.
\begin{thm}[Skinner--Urban] \label{thm:skinner-urban}
Assume $E$ has good ordinary reduction at $p$.
If $E[p]$ is irreducible and there exists a prime $\ell$ exactly dividing $N$ such that $E[p]$ is ramified at $\ell$, then
$$( L_p(E) ) \supseteq \mathrm{char}_{\Lambda}( \mathrm{Sel}(\mathbb{Q}_\infty, E[p^\infty])^\vee )$$
as ideals of $\Lambda$.
\end{thm}
Therefore, Conjecture \ref{conj:imc-mazur-greenberg} is confirmed for a large class of elliptic curves with good ordinary reduction, and we now illustrate its application to Birch and Swinnerton-Dyer conjecture.
\begin{align*}
L(E, 1) = 0 & \Leftrightarrow \mathbf{1}(L_p(E)) = 0  & \textrm{(\ref{eqn:interpolation-trivial-character})} \\
 &  \Leftrightarrow \mathrm{length} \left( \mathrm{Sel}(\mathbb{Q}_\infty, E[p^\infty])^\vee \right)_{\Gamma} = \infty & \textrm{(\ref{eqn:imc-mazur-greenberg})} \\
 &  \Leftrightarrow \mathrm{length} \left( \mathrm{Sel}(\mathbb{Q}_\infty, E[p^\infty])^{\Gamma} \right)^\vee = \infty \\
 &  \Leftrightarrow \mathrm{length} \left( \mathrm{Sel}(\mathbb{Q}, E[p^\infty]) \right)^\vee = \infty & \textrm{Prop. \ref{prop:control-thoerem}}
\end{align*}
This equivalence strengthens Theorem \ref{thm:kato-finiteness-mordell-weil-sha} to the if and only if statement.
\begin{rem}
We actually need the Iwasawa main conjecture \emph{localized at the augmentation ideal} only.
This type of idea generalizes the Selmer rank one case, which is called the rank one $p$-converse to the theorem of Gross--Zagier and Kolyvagin. See \cite{kim-p-converse} for example.
See also \cite{skinner-converse, wei-zhang-mazur-tate, burungale-tian-p-converse} for the development of the rank one $p$-converse results.
\end{rem}
\subsubsection{}
We briefly discuss the $p$-adic Birch and Swinnerton-Dyer conjecture following \cite{mtt}, \cite[$\S$18]{kato-euler-systems}, and \cite{kato-icm-2002}.
\begin{conj}[$p$-adic BSD]
Assume $E$ has good ordinary reduction at $p$. Then
$$\mathrm{cork}_{\mathbb{Z}_p} \mathrm{Sel}(\mathbb{Q}, E[p^\infty])  = \mathrm{ord}_{\chi = \mathbf{1}} L_p(E) .$$
\end{conj} 
Kato proved the following inequality which bounds Selmer ranks (so Mordell--Weil ranks) \cite[Thm. 18.4]{kato-euler-systems}.
\begin{cor}[Kato]
$$\mathrm{cork}_{\mathbb{Z}_p} \mathrm{Sel}(\mathbb{Q}, E[p^\infty]) \leq \mathrm{ord}_{\chi = \mathbf{1}} L_p(E) .$$
\end{cor} 
In fact, the formulation and the result both extend to the finite slope case, but the exceptional zero occurs for the multiplicative reduction case.
Unlike the results on the classical Birch and Swinnerton-Dyer conjecture, this result does not have any low rank restriction.

\subsection{From Kato's zeta elements to modular elements: {\'{e}}tale construction I}
In this subsection, we study the dotted arrow in the diagram below
\[
\xymatrix{
\textrm{Kato's zeta elements} \ar[d]_-{\mathrm{exp}^*} \ar@{-->}[rd] \ar[r]^-{\mathrm{Col}} &  \textrm{$p$-adic $L$-functions} \\
\textrm{modular symbols/$L$-values} \ar[r]_-{\textrm{Stickelberger}} & \textrm{modular elements} \ar[u]_-{\textrm{$p$-stabilization}+\varprojlim}  \\
}
\]
following the work of Kurihara \cite{kurihara-invent}.
\subsubsection{}
We discuss a little bit of $p$-adic Hodge theory.
Fix a compatible system $(\zeta_{p^n})_n \in \mathbb{Z}_p(1)$ of primitive $p$-power roots of unity.
Let $D = \mathbf{D}_{\mathrm{dR}}(V) = (V \otimes \mathbf{B}_{\mathrm{dR}})^{G_{\mathbb{Q}_p}}$ be the filtered $\varphi$-module of rank two associated to $V$
and
$D^0  = \mathbf{D}^0_{\mathrm{dR}}(V) = \mathbb{Q}_p \cdot \omega_E \subseteq D$ the $\varphi$-stable submodule of rank one
where $\mathbf{B}_{\mathrm{dR}}$ is the de Rham period ring and $\omega_E$ is the invariant N\'{e}ron differential of $E$.
Then $D^0$ is the cotangent space of $E$ (at the identity) and $D/D^0$ is the tangent space of $E$.
Furthermore, the Bloch--Kato exponential map
$$\mathrm{exp} : D \otimes \mathbb{Q}_p(\zeta_{p^{n+1}}) \to   D/D^0 \otimes \mathbb{Q}_p(\zeta_{p^{n+1}}) \to \mathrm{H}^1( \mathbb{Q}_p(\zeta_{p^{n+1}}), V)$$
extends the composition of the Kummer map and the formal Lie group exponential map
$$\widehat{\mathbb{G}}_a(  \mathbb{Q}_p(\zeta_{p^{n+1}}) ) \otimes \mathbb{Q}_p
\to 
\widehat{E}(  \mathbb{Q}_p(\zeta_{p^{n+1}}) ) \otimes \mathbb{Q}_p
\to
\mathrm{H}^1( \mathbb{Q}_p(\zeta_{p^{n+1}}), V)$$
with identification
$\widehat{\mathbb{G}}_a(  \mathbb{Q}_p(\zeta_{p^{n+1}}) ) \otimes \mathbb{Q}_p = \mathbb{Q}_p(\zeta_{p^{n+1}}) \simeq D/D^0 \otimes \mathbb{Q}_p(\zeta_{p^{n+1}}) = \mathbb{Q}_p(\zeta_{p^{n+1}})\omega^*_E$
as in \S\ref{subsubsec:local-galois-cohomology}.
\subsubsection{}
We follow \cite[\S3]{kurihara-invent} and \cite[\S1]{kurihara-pollack} but with slight modifications.
Consider an element
\begin{align*}
\gamma_n(\varphi^{n+2}(\omega_E)) & = \sum^n_{i = 0} \varphi^{i+1}(\omega_E) \otimes \zeta_{p^{n+1-i}} + (1-\varphi)^{-1} \varphi^{n+2}(\omega_E) \\
& \in D \otimes \mathbb{Q}_p(\zeta_{p^{n+1}})
\end{align*}
and a homomorphism
$$P_n : \mathrm{H}^1(\mathbb{Q}_p(\zeta_{p^{n+1}}), T) \to \mathbb{Q}_p[G] = \mathbb{Q}_p[\mathrm{Gal}(\mathbb{Q}_p(\zeta_{p^{n+1}})/\mathbb{Q}_p)] \simeq \mathbb{Q}_p[\mathrm{Gal}(\mathbb{Q}(\zeta_{p^{n+1}})/\mathbb{Q})]$$
sending
$a$ to
$$\dfrac{1}{[\varphi(\omega_E), \omega_E]_{\mathrm{dR}}} \cdot \sum_{\sigma \in G} \langle \mathrm{exp}(\gamma_n(\varphi^{n+2}(\omega_E))^\sigma), a \rangle_p \cdot \sigma$$
where 
\[
\xymatrix{
\gamma_n(\varphi^{n+2}(\omega_E)) \in D \otimes \mathbb{Q}_p(\zeta_{p^{n+1}}) \ar[d]^-{\mathrm{exp}} & \times & D \otimes \mathbb{Q}_p(\zeta_{p^{n+1}})  \ar[r]^-{[-,-]_{\mathrm{dR}}} & \mathbb{Q}_p(\zeta_{p^{n+1}}) \ar[d]^-{\mathrm{Tr}}\\
\mathrm{exp}(\gamma_n(\varphi^{n+2}(\omega_E))) \in \mathrm{H}^1(\mathbb{Q}_p(\zeta_{p^{n+1}}), V)   & \times & a \in \mathrm{H}^1(\mathbb{Q}_p(\zeta_{p^{n+1}}), V) \ar[u]_-{\mathrm{exp}^*} \ar[r]^-{\langle -, - \rangle_p}  & \mathbb{Q}_p .
}
\]
\begin{prop}
If $E[p]$ is irreducible, then
$$\mathrm{Im}(P_n) \subseteq \mathbb{Z}_p[\mathrm{Gal}(\mathbb{Q}(\zeta_{p^{n+1}})/\mathbb{Q})] .$$
\end{prop}
\begin{proof}
See \cite[Prop. 3.6]{kurihara-invent}. 
Note that our definition of $\gamma_n$ is Kurihara's $\gamma_n$ multiplied by $p^n$.
\end{proof}
\begin{thm}[Kurihara]
We have explicit formula
$$P_n(z_{\mathbb{Q}(\zeta_{p^{n+1}})}) = \theta_{\mathbb{Q}(\zeta_{p^{n+1}})}.$$
If $E[p]$ is irreducible, then $\theta_{\mathbb{Q}_n}$ is the image of $P_n(z_{\mathbb{Q}(\zeta_{p^{n+1}})})$ in 
$\mathbb{Z}_p[\mathrm{Gal}(\mathbb{Q}_n/\mathbb{Q})]$.
\end{thm}
\begin{proof}
See \cite[Lem. 7.2]{kurihara-invent}.
\end{proof}
The above theorem should be considered as the \emph{\'{e}tale} construction of Mazur--Tate's modular elements.
\subsection{Coleman maps: {\'{e}}tale construction II} \label{subsec:coleman-maps}
Various constructions of $p$-adic $L$-functions have similarities with the realizations of motives\footnote{I first learned this picture from Shanwen Wang's talk at Sendai International Conference on Arithmetic Geometry in 2016.}
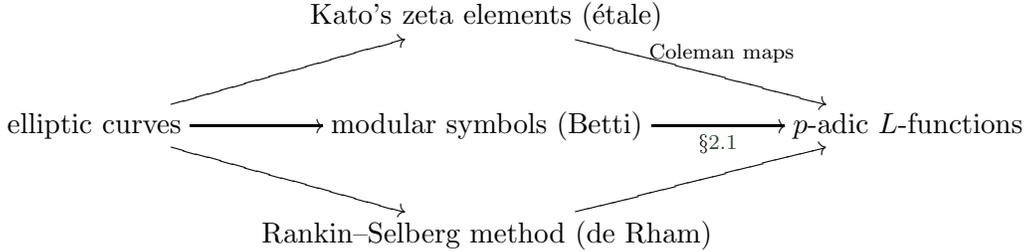
\[
\xymatrix{
& \textrm{Kato's zeta elements (\'{e}tale)} \ar[rd]^-{\textrm{Coleman maps}} \\
\textrm{elliptic curves} \ar[ru] \ar[r] \ar[rd] & \textrm{modular symbols (Betti)} \ar[r]_-{\textrm{$\S$\ref{subsec:betti-construction}}} & \textrm{$p$-adic $L$-functions} \\
& \textrm{Rankin--Selberg method (de Rham)} \ar[ru]
}
\]
and the Coleman maps play the central role in the \'{e}tale construction.
The Rankin--Selberg method construction of $p$-adic $L$-functions can be found in the work of Hida, Panchishkin, and T. Fukaya \cite{hida-invent-1981,panchishkin-2003,fukaya-coleman-power-series}.

Let $\mathrm{H}^1_{\mathrm{Iw}}(\mathbb{Q}_p, T) = \varprojlim_n \mathrm{H}^1(\mathbb{Q}_{n,p}, T)$ be the local Iwasawa cohomology of $T$ where $\mathbb{Q}_{n,p}$ is the $p$-adic completion of $\mathbb{Q}_{n}$.
\subsubsection{The ordinary case}
Assume that $E$ has good ordinary reduction at $p$.
\begin{thm}
There exists the Coleman map
$$\mathrm{Col} : \mathrm{H}^1_{\mathrm{Iw}}(\mathbb{Q}_p, T) \to \Lambda$$
satisfying
$\mathrm{Col}(\mathrm{loc}_p ( z_{\mathbb{Q}_{\infty}} )) = L_p(E)$.
\end{thm}
\begin{proof}
See \cite[Appendix]{rubin-es-mec} for the explicit construction.
\end{proof}

\subsubsection{The supersingular case}
Assume that $E$ has good supersingular reduction at $p \geq 5$. See Sprung \cite{sprung-ap-nonzero} for the $p \leq 3$ case.
\begin{thm}[Kobayashi]
There exists the signed Coleman maps
$$\mathrm{Col}^{\pm} : \mathrm{H}^1_{\mathrm{Iw}}(\mathbb{Q}_p, T) \to \Lambda$$
satisfying
$\mathrm{Col}^{\pm}(\mathrm{loc}_p ( z_{\mathbb{Q}_{\infty}} ) ) = L^{\pm}_p(E)$
\end{thm}
\begin{proof}
See \cite{kobayashi-thesis}.
\end{proof}

\subsubsection{Remarks on the construction}
By using the formal logarithm, we have the map
$$\widehat{E}(  \mathbb{Q}_p(\zeta_{p^{n+1}}) ) \otimes \mathbb{Q}_p
 \to 
\widehat{\mathbb{G}}_a(  \mathbb{Q}_p(\zeta_{p^{n+1}}) ) \otimes \mathbb{Q}_p
\to 
D/D^0 \otimes \mathbb{Q}_p(\zeta_{p^{n+1}}).$$
In particular, 
$\gamma_n(\varphi^{n+2}(\omega_E)) \pmod{D^0}$ can be realized as local points of elliptic curves (or modular abelian varieties).
In order to construct Coleman maps for such cases, it is natural to study local points in the formal groups.
For the case of higher weight modular forms, such a local point description is not available, so we can only use $p$-adic Hodge theory following Perrin-Riou's strategy \cite{perrin-riou-local-iwasawa}. See also \cite[Thm. 16.4]{kato-euler-systems}. Kurihara's construction of modular elements also generalizes to modular forms of higher weight \cite{kim-kato,epw2}.

\subsection{The global duality argument and Kato's main conjecture} \label{subsec:global-duality}
The global Poitou--Tate duality yields the exact sequence
\[
\xymatrix{
\mathrm{H}^1(\mathbb{Q}, T) \ar[r]^-{\mathrm{loc}^s_p} & \mathrm{H}^1_{/f}(\mathbb{Q}_p, T) \ar[d]^-{\mathrm{exp}^*} \ar[r] & \mathrm{Sel}(\mathbb{Q}, E[p^\infty])^\vee \ar[r] & \mathrm{Sel}_0(\mathbb{Q}, E[p^\infty])^\vee \\
 & \mathrm{H}^0(E/\mathbb{Q}_p, \Omega^1)
}
\]
and Kato's zeta element 
$z_{\mathbb{Q}} \in \mathrm{H}^1(\mathbb{Q}, T)$
maps to
$$\mathrm{exp}^* \circ \mathrm{loc}^s_p ( z_{\mathbb{Q}} ) = \dfrac{L^{(p)}(E,1)}{\Omega^+_E} \cdot \omega_E \in \mathbb{Q}_p\omega_E = D^0 = \mathrm{H}^0(E/\mathbb{Q}_p, \Omega^1) .$$
Here, $\mathrm{Sel}_0$ means the $p$-strict Selmer group.
We interpolate this picture over the cyclotomic tower.

\subsubsection{Kato's main conjecture}
Recall that $$z_{\mathbb{Q}_{\infty}} = \varprojlim_{n} z_{\mathbb{Q}_{n}} \in \mathrm{H}^1_{\mathrm{Iw}}(\mathbb{Q}, T) = \varprojlim_{n}
\mathrm{H}^1(\mathbb{Q}_n, T)$$
is Beilinson--Kato's zeta element over $\mathbb{Q}_{\infty}$.
The formulation of Kato's main conjecture \cite[Conj. 12.10]{kato-euler-systems} and the one-sided divisibility result \cite[Thm. 12.5.(4)]{kato-euler-systems}  are insensitive to the reduction type. 
\begin{thm}[Kato] \label{thm:kato-divisibility-sel0}
Assume that the image of $\rho$ contains a conjugate of $\mathrm{SL}_2(\mathbb{Z}_p)$.
\begin{enumerate}
\item $\mathrm{Sel}_0(\mathbb{Q}_{\infty}, E[p^\infty])$ is $\Lambda$-cotorsion.
\item $\mathrm{char}_{\Lambda} \dfrac{\mathrm{H}^1_{\mathrm{Iw}}(\mathbb{Q}, T)}{\Lambda z_{\mathbb{Q}_{\infty}} } \subseteq \mathrm{char}_{\Lambda} \mathrm{Sel}_0(\mathbb{Q}_\infty, E[p^\infty])^\vee$.
\end{enumerate}
\end{thm}
\begin{conj}[Kato's main conjecture] \label{conj:kato-main-conj-sel0}
Assume that the image of $\rho$ contains a conjugate of $\mathrm{SL}_2(\mathbb{Z}_p)$. Then
$$\mathrm{char}_{\Lambda} \dfrac{\mathrm{H}^1_{\mathrm{Iw}}(\mathbb{Q}, T)}{\Lambda z_{\mathbb{Q}_{\infty}} } = \mathrm{char}_{\Lambda}
\mathrm{Sel}_0(\mathbb{Q}_\infty, E[p^\infty])^\vee.$$
\end{conj}

\subsubsection{The good ordinary case}
When $E$ has good ordinary reduction at $p$, the local Galois representation at $p$ is reducible, so we have exact sequence of $G_{\mathbb{Q}_p}$-modules
\[
\xymatrix{
0 \ar[r] & T^0 \ar[r] &  T \ar[r] &  T^1 \ar[r] & 0 .
}
\]
Then the following exact sequence follows from the global duality argument
\[
\xymatrix{
\dfrac{\mathrm{H}^1_{\mathrm{Iw}}(\mathbb{Q}, T)}{\Lambda z_{\mathbb{Q}_{\infty}} } \ar@{^{(}->}[r]^-{\mathrm{loc}^s_p} & 
\dfrac{\mathrm{H}^1_{\mathrm{Iw}}(\mathbb{Q}_p, T)}{ \langle \mathrm{loc}^s_p z_{\mathbb{Q}_{\infty}} \rangle + \mathrm{H}^1_{\mathrm{Iw}}(\mathbb{Q}_p, T^0)}
 \ar[d]^-{\mathrm{Col}}_-{\simeq} \ar[r] & \mathrm{Sel}(\mathbb{Q}_\infty, E[p^\infty])^\vee \ar@{->>}[r] & \mathrm{Sel}_0(\mathbb{Q}_\infty, E[p^\infty])^\vee \\
 & \Lambda / L_p(E)
}
\]
Therefore, Conjecture \ref{conj:imc-mazur-greenberg} and Conjecture \ref{conj:kato-main-conj-sel0} are equivalent.
See \cite[$\S$17.13]{kato-euler-systems} for the detail.
\subsubsection{The good supersingular case}
When $E$ has good supersingular reduction at p and $a_p(E) = 0$ (e.g. $p >3$), we use Kobayashi--Pollack's $\pm$-Iwasawa theory \cite{kobayashi-thesis, pollack-thesis}.
Then the following exact sequence follows from the global duality argument
\[
\xymatrix{
\dfrac{\mathrm{H}^1_{\mathrm{Iw}}(\mathbb{Q}, T)}{\Lambda z_{\mathbb{Q}_{\infty}} } \ar@{^{(}->}[r]^-{\mathrm{loc}^s_p} & 
\dfrac{\mathrm{H}^1_{\mathrm{Iw}}(\mathbb{Q}_p, T)}{ \langle \mathrm{loc}^s_p z_{\mathbb{Q}_{\infty}} \rangle + \mathrm{ker}(\mathrm{Col}^{\pm})}
 \ar[d]^-{\mathrm{Col}^{\pm}}_-{\simeq} \ar[r] & \mathrm{Sel}^{\pm}(\mathbb{Q}_\infty, E[p^\infty])^\vee \ar@{->>}[r] & \mathrm{Sel}_0(\mathbb{Q}_\infty, E[p^\infty])^\vee \\
 & \Lambda / L^{\pm}_p(E)
}
\]
Therefore, Kobayashi's signed main conjecture and Conjecture \ref{conj:kato-main-conj-sel0} are equivalent.
See \cite[Thm. 7.4]{kobayashi-thesis} for the detail.

\section{Lecture 3: The Euler and Kolyvagin system machinery}
The main theme of the third lecture is to explain how to develop the connection between the size/structure of arithmetically interesting modules (ideal class groups, Selmer groups, Tate--Shafarevich groups, $\cdots$) and the special values of $L$-functions.
See \cite{rubin-book,mazur-rubin-book,loeffler-zerbes-arizona} for the detail.
\subsection{A rough description}
\subsubsection{}
We start with a very informal and imprecise definition of Euler systems. In fact, this shows what we expect to see from Euler systems.
\begin{defn}[informal and imprecise]
An Euler system attached to a $p$-adic representation $V$ is a cohomology class
$$c =c_{\mathbb{Q}} \in \mathrm{H}^1(\mathbb{Q}, V)$$
with properties:
\begin{enumerate}
\item it can be related to $L(V^*(1),0)$ or to $p$-adic $L$-functions (i.e. $L$-values in families)
\item it can be ``deformed" $p$-adically or adelically. 
\end{enumerate}
The former condition is called the \emph{explicit reciprocity law} and the latter one is called the \emph{Euler system norm relation}. 
All the Euler factors appear in the second one.
\end{defn}
Then the Euler system \emph{machinery} yields the following theorem.
\begin{thm}[Euler system machinery]
If there exists a non-zero Euler system and the image of the Galois representation is large enough, then the Selmer group of the dual representation is finite.
\end{thm}
\begin{proof}
See \cite{rubin-book, kato-euler-iwasawa-selmer, perrin-riou-euler-systems}.
\end{proof}
More precise versions of the above statements will follow (Definition \ref{defn:euler-systems-rubin} and Theorem \ref{thm:finiteness-rubin}).
\subsubsection{} \label{subsubsec:image}
There are three layers of the applications of Euler systems:
\begin{enumerate}
\item Bounding the exponent of Selmer groups (also known as the annihilation result).
This ensures the finiteness but it does not give a precise upper bound.
\item Bounding the size of Selmer groups.
This gives an upper bound of the Selmer group by the bottom class of the Euler system.
\item Describing the structure of Selmer groups.
This determines the structure of the Selmer group by the associated Kolyvagin system.
\end{enumerate}
In order to have better applications, more strict conditions on the Galois image are needed.
This is essentially related to the use of Chebotarev density theorem.

We describe the corresponding Galois image assumptions for the case of elliptic curves.
\begin{enumerate}
\item[($\mathrm{Hyp}(\mathbb{Q}, V)$)] $V$ is irreducible, and there exists $\tau \in \mathrm{Gal}(\overline{\mathbb{Q}}/\mathbb{Q}(\zeta_{p^\infty}))$ such that $V/(\tau-1)V$ is one-dimensional \cite[page 37]{rubin-book}. This is enough for bounding the exponent of Selmer groups.
\item[($\mathrm{Hyp}(\mathbb{Q}, T)$)] $p > 2$, $E[p]$ is irreducible, and there exists $\tau \in \mathrm{Gal}(\overline{\mathbb{Q}}/\mathbb{Q}(\zeta_{p^\infty}))$ such that $T/(\tau-1)T$ is one-dimensional \cite[page 37]{rubin-book}. This is enough for bounding the size of Selmer groups.
\item[(KS)] $p > 5$, $E[p]$ is absolutely irreducible, there exists $\tau \in \mathrm{Gal}(\overline{\mathbb{Q}}/\mathbb{Q}(\zeta_{p^\infty}))$ such that $T/(\tau-1)T$ is one-dimensional, and $\mathrm{H}^1(\mathbb{Q}_{T}(\zeta_{p^\infty})/\mathbb{Q}, E[p]) = 0$
where $\mathbb{Q}_T$ is the minimal field extension of $\mathbb{Q}$ such that $\mathrm{Gal}(\overline{\mathbb{Q}}/\mathbb{Q}_T)$ acts trivially on $T$ \cite[page 27]{mazur-rubin-book}. This is enough for describing the structure of Selmer groups.
\end{enumerate}
Recently, the $p >5$ condition in (KS) is weakened to $p> 2$ by the work of Sakamoto \cite{sakamoto-p-3}. 
Our focus lies in the latter two statements.
\subsection{Setup}
Let $p > 2$ be a prime and $F$ be a finite extension of $\mathbb{Q}_p$.
Let $\mathcal{O} = \mathcal{O}_F$, and denote by $\mathfrak{m}_F = (\varpi)$ the maximal ideal of $\mathcal{O}$.
Let $T$ be a free $\mathcal{O}$-module of finite rank $n$ with continuous action of $G_{\mathbb{Q}}$.
Let $V =  T \otimes_{\mathcal{O}} F$.
We assume that $V$ is geometric, i.e.
\begin{itemize}
\item $V$ is unramified outside a finite set of primes $\Sigma$
\item $V$ is de Rham at $p$ in the sense of Fontaine's $p$-adic Hodge theory.
\end{itemize}
This means that $V$ is ``reasonably good".
Then $\rho : G_{\mathbb{Q}} \to \mathrm{Aut}_{\mathcal{O}}(V)$
factors through $G_{\mathbb{Q}, \Sigma}= \mathrm{Gal}(\mathbb{Q}_{\Sigma}/\mathbb{Q})$
where $\mathbb{Q}_{\Sigma}$ is the maximal extension of $\mathbb{Q}$ unramified outside $\Sigma$.
We also write $W = V/T$. When $T = \mathrm{Ta}_pE$, we have $W  \simeq E[p^\infty]$.

For $m \geq 1$, we define $W_m = \varpi^{-m}T/T = W[\varpi^m] \simeq T/\varpi^m T$.
Define their Kummer duals by
\[
\xymatrix{
T^*(1) = \mathrm{Hom}_{\mathcal{O}}(T, \mathcal{O}(1)) , &
V^*(1) = \mathrm{Hom}_{F}(V, F(1)) , &
W^*(1) = V^*(1) / T^*(1) .
}
\]

\subsection{$L$-functions and Selmer groups}
For $\ell \in\Sigma$, let
$$P_\ell (V,x) = \mathrm{det}(I_n - x \cdot \rho(\mathrm{Frob}_\ell)\vert_V)$$
where $\mathrm{Frob}_\ell$ is the arithmetic Frobenius\footnote{Kato follows the \emph{cohomological} convention of modular Galois representations but uses \emph{arithmetic} Frobenius elements \cite[$\S$14.10]{kato-euler-systems}.} at $\ell$.
Set 
$$L^{\Sigma} (V,s) = \prod_{\ell \not\in \Sigma} P(V, \ell^{-s})^{-1}$$
which converges for $\mathrm{Re}(s) \gg 0$.

We recall the notion of Selmer structures/local conditions.
\subsubsection{}
For every prime $\ell$ except $p$ and $ \infty$, define
$\mathrm{H}^1_f(\mathbb{Q}_\ell, V) = \mathrm{ker} \left( \mathrm{H}^1(\mathbb{Q}_\ell, V) \to  \mathrm{H}^1(I_\ell, V)   \right )$
 where $I_\ell \subseteq \mathrm{Gal}(\overline{\mathbb{Q}}_\ell/\mathbb{Q}_\ell)$ is the inertia subgroup at $\ell$.
\subsubsection{}
For $\ell = \infty$, we have $\mathrm{H}^1_f (\mathbb{R}, V) = 0$ since $p > 2$.
\subsubsection{}
For $\ell = p$, we consider two different structures.
\begin{itemize}
\item $\mathrm{H}^1_f(\mathbb{Q}_p, V) = 0$ for the $p$-strict Selmer groups.
\item $\mathrm{H}^1_f(\mathbb{Q}_p, V) = \mathrm{ker} \left( \mathrm{H}^1(\mathbb{Q}_p, V) \to \mathrm{H}^1(\mathbb{Q}_p, V \otimes \mathbf{B}_{\mathrm{cris}}) \right)$ for the Bloch--Kato Selmer groups.
\end{itemize}
\subsubsection{}
In any case, 
$ \mathrm{H}^1_f(\mathbb{Q}_\ell, T)$ and
 $\mathrm{H}^1_f(\mathbb{Q}_\ell, W)$
 are defined as the preimage and the image of $\mathrm{H}^1_f(\mathbb{Q}_\ell, V)$
 with respect to $T \to V \to W$, respectively.
Write $\mathrm{H}^1_{/f} = \dfrac{\mathrm{H}^1}{\mathrm{H}^1_f}$.
\subsubsection{}
Let $\Sigma'$ be a finite set of primes and $M$ be a Galois module.
Then the $\Sigma'$-relaxed Selmer group of $M$ is defined by
\begin{align*}
\mathrm{Sel}^{\Sigma'}(\mathbb{Q}, M) & = 
\mathrm{ker} \left(
\mathrm{H}^1(\mathbb{Q}, M) \to \prod_{\ell \not\in \Sigma'} \mathrm{H}^1_{/f}(\mathbb{Q}_\ell, M)
 \right) \\
& = \mathrm{ker} \left(
\mathrm{H}^1(\mathbb{Q}_{\Sigma \cup \Sigma'}, M) \to \prod_{\ell \not\in (\Sigma \cup \Sigma') \setminus \Sigma'} \mathrm{H}^1_{/f}(\mathbb{Q}_\ell, M)
 \right)
\end{align*}
and the $\Sigma'$-strict Selmer group of $M$ is defined by
$$\mathrm{Sel}_{\Sigma'}(\mathbb{Q}, M)  = 
\mathrm{ker} \left(
\mathrm{Sel}^{\Sigma'}(\mathbb{Q}, M) \to \prod_{\ell \in \Sigma'} \mathrm{H}^1(\mathbb{Q}_\ell, M)
 \right) .$$
\subsubsection{}
The weak form of the Bloch--Kato conjecture can be stated as follows \cite{bellaiche-cmi-notes}:
\begin{conj}[Bloch--Kato]
$$\mathrm{ord}_{s=0} L(V,s) = \mathrm{dim}_F \mathrm{Sel}(\mathbb{Q}, V^*(1)) - \mathrm{dim}_F \mathrm{H}^0(\mathbb{Q}, V^*(1)) .$$
\end{conj} 
 
\subsection{The Euler system argument: a sketch of the idea}
Let $\mathbb{Q}^{\mathrm{ab}}$ be the maximal abelian extension of $\mathbb{Q}$.
For a prime $q$ with $q \equiv 1 \pmod{p}$, denote by $\mathbb{Q}(q)$ the maximal $p$-subextension of $\mathbb{Q}$ in $\mathbb{Q}(\zeta_q)$.
Let $\mathcal{K}$ be a subextension of $\mathbb{Q}$ in $\mathbb{Q}^{\mathrm{ab}}$
such that $\mathbb{Q}_\infty \subseteq \mathcal{K}$ and $\mathbb{Q}(q) \subseteq \mathcal{K}$ for infinitely many primes $q$ with $q \equiv 1 \pmod{p}$.
\begin{defn} \label{defn:euler-systems-rubin}
An \textbf{Euler system for $(T, \mathcal{K})$}
is a collection of cohomology classes
$$\mathbf{z} = \left\lbrace z_K \in \mathrm{H}^1(K,T) : \mathbb{Q} \subseteq K \subseteq \mathcal{K} \right\rbrace$$
where $K$ runs over \emph{finite} extensions of $\mathbb{Q}$ in $\mathcal{K}$
such that
$$
\mathrm{cores}_{\mathbb{Q}(\zeta_{n\ell})/ \mathbb{Q}(\zeta_{n})  } (z_{ \mathbb{Q}(\zeta_{n\ell})  }  ) = \left\lbrace
\begin{array}{ll}
 z_{ \mathbb{Q}(\zeta_{n})  } & \textrm{ if  } \ell \vert n \textrm{ or } \ell \in \Sigma \\
 P_\ell (V^*(1), \mathrm{Fr}^{-1}_\ell) \cdot z_{ \mathbb{Q}(\zeta_{n})  } & \textrm{ otherwise.}
\end{array} \right.
$$
\end{defn} 
As a rough picture, $z_{ \mathbb{Q}(\zeta_{n})}$ is related to $L(V \otimes \mathbb{Q}(\zeta_n),0)$
and $\mathrm{cores}_{\mathbb{Q}(\zeta_{n})/ \mathbb{Q}  } (z_{ \mathbb{Q}(\zeta_{n})})$ is related to $L^{\Sigma_n}(V,0)$
where $\Sigma_n = \Sigma \cup \lbrace \ell \vert n \rbrace$.
\begin{thm}[Rubin] \label{thm:finiteness-rubin}
Let $\mathbf{z}$ be an Euler system for $(T, \mathcal{K})$.
Assume $\mathrm{Hyp} (\mathbb{Q},T)$ in $\S$\ref{subsubsec:image} holds.
If $z_{\mathbb{Q}}$ is not a torsion, then
$\mathrm{Sel}_{\lbrace p \rbrace} (\mathbb{Q}, W^*(1))$ is finite.
\end{thm}
We sketch the idea of the proof of Theorem \ref{thm:finiteness-rubin}.
We bound $ \mathrm{Sel}_{\lbrace p \rbrace} (\mathbb{Q}, W^*_m(1))$ independently of $m$,  and the main tools are local and global dualities in Galois cohomology.

Let $$\langle - ,  - \rangle_\ell : 
 \mathrm{H}^1(\mathbb{Q}_\ell, W_m) \times
\mathrm{H}^1(\mathbb{Q}_\ell, W^*_m(1)) \to \mathcal{O}/\varpi^m\mathcal{O}$$
be the local Tate pairing.
Then the local duality says
$\mathrm{H}^1_f(\mathbb{Q}_\ell, W_m)^\perp = \mathrm{H}^1_f(\mathbb{Q}_\ell, W^*_m(1))$.
Fix a finite set of primes $\Sigma'$ not including $p$.
Consider the diagram
\begin{equation} \label{eqn:global-duality-Sigma}
\begin{split}
\xymatrix{
\mathrm{Sel}^{ \lbrace p \rbrace }( \mathbb{Q}, W_m ) \ar@{^{(}->}[r] & \mathrm{Sel}^{ \Sigma' \cup \lbrace p \rbrace }( \mathbb{Q}, W_m ) \ar[r]^-{\mathrm{loc}^s_{\Sigma'}} & \sum_{\ell \in \Sigma'} \mathrm{H}^1_{/f}(\mathbb{Q}_\ell, W_m) \\
& & \times \\
\mathrm{Sel}_{\Sigma' \cup \lbrace p \rbrace }( \mathbb{Q}, W^*_m(1) ) \ar@{^{(}->}[r] & \mathrm{Sel}_{  \lbrace p \rbrace }( \mathbb{Q}, W^*_m(1) ) \ar[r]^-{\mathrm{loc}^f_{\Sigma'}} & \sum_{\ell \in \Sigma'} \mathrm{H}^1_f(\mathbb{Q}_\ell, W^*_m(1)) \ar[d]_-{ \sum_{\ell \in \Sigma'} \langle - , -  \rangle_{\ell} } \\
 & & \mathcal{O}/\varpi^m \mathcal{O}
}
\end{split}
\end{equation}
where
\begin{align*}
\mathrm{loc}^s_{\Sigma'} & = \bigoplus_{\ell \in \Sigma'} \mathrm{loc}^s_{\ell} : \mathrm{Sel}^{ \Sigma' \cup \lbrace p \rbrace }( \mathbb{Q}, W_m ) \to \bigoplus_{\ell \in \Sigma'} \mathrm{H}^1(\mathbb{Q}_\ell, W_m) \to \bigoplus_{\ell \in \Sigma'} \mathrm{H}^1_{/f}(\mathbb{Q}_\ell, W_m) , \\
\mathrm{loc}^f_{\Sigma'} & = \bigoplus_{\ell \in \Sigma'} \mathrm{loc}^f_{\ell} : \mathrm{Sel}_{  \lbrace p \rbrace }( \mathbb{Q}, W^*_m(1) ) \to \bigoplus_{\ell \in \Sigma'} \mathrm{H}^1_f(\mathbb{Q}_\ell, W^*_m(1)) .
\end{align*}
The global duality implies $\mathrm{im}( \mathrm{loc}^s_{\Sigma'} )^{\perp} =   \mathrm{im}( \mathrm{loc}^f_{\Sigma'} )$.
One of the key intuitions of the Euler system argument is the following behavior.
\begin{quote}
If $\mathrm{im}( \mathrm{loc}^s_{\Sigma'} )$ gets larger, then $\mathrm{im}( \mathrm{loc}^f_{\Sigma'} )$ gets smaller.
\end{quote}
Therefore, it suffices to construct ``relevant" $\Sigma'$ and elements $\kappa_{\Sigma'} \in \mathrm{Sel}^{ \Sigma' \cup \lbrace p \rbrace }( \mathbb{Q}, W_m )$ from $\mathbf{z}$
such that
\begin{itemize}
\item $\kappa_{\Sigma'}$ is ramified at primes in $\Sigma'$, so its image under $\mathrm{loc}^s_{\Sigma'}$ is non-trivial, and
\item $\mathrm{length} ( \mathrm{coker}( \mathrm{loc}^{s}_{\Sigma'} ) )$
and 
$\mathrm{length} ( \mathrm{Sel}_{ \Sigma' \cup \lbrace p \rbrace } (\mathbb{Q}, W^*_m(1))  )$
are bounded independently of $m$.
\end{itemize}
Here,  $\kappa_{\Sigma'}$'s are called \textbf{Kolyvagin derivative classes}.
\begin{rem}
Also, ``relevant" $\Sigma'$ means that
the image of the arithmetic Frobenius $\mathrm{Fr}_{\ell}$ at $\ell \in \Sigma'$
is equivalent to the image of $\tau \in \mathrm{Gal}(\mathbb{Q}(W_m)(\zeta_{p^m})/\mathbb{Q})$.
Chebotarev density theorem plays the key role to construct such a $\Sigma'$.
\end{rem}
Therefore, the finiteness of the $p$-strict Selmer group $\mathrm{Sel}_{ \lbrace p \rbrace }( \mathbb{Q}, W^*(1) )$ can be proved.
\begin{assu}
From now on, we consider the case of elliptic curves only.
\end{assu}
Then how to obtain Theorem \ref{thm:kato-finiteness-mordell-weil-sha} from the finiteness of  $\mathrm{Sel}_{ \lbrace p \rbrace }( \mathbb{Q}, W^*(1) )$?
We now compare the difference between $\mathrm{Sel}_{ \lbrace p \rbrace }( \mathbb{Q}, W^*(1) )$ and $\mathrm{Sel}( \mathbb{Q}, W^*(1) )$.
We consider the following variant of (\ref{eqn:global-duality-Sigma}) with the Bloch--Kato local condition at $p$. In other words, we have
\[
\xymatrix{
\mathrm{Sel}( \mathbb{Q}, T ) \ar@{^{(}->}[r] & \mathrm{Sel}^{  \lbrace p \rbrace }( \mathbb{Q}, T ) \ar[r]^-{\mathrm{loc}^s_{p}} &  \mathrm{H}^1_{/f}(\mathbb{Q}_p, T) \\
& & \times \\
\mathrm{Sel}_{ \lbrace p \rbrace }( \mathbb{Q}, W^*(1) ) \ar@{^{(}->}[r] & \mathrm{Sel}( \mathbb{Q}, W^*(1) ) \ar[r]^-{\mathrm{loc}^f_{p}} & \mathrm{H}^1_f(\mathbb{Q}_p, W^*(1)) \ar[d]_-{  \langle - , -  \rangle_{p} } \\
 & & F/ \mathcal{O}
}
\]
with $\mathrm{im}(\mathrm{loc}^s_p)^\perp = \mathrm{im}(\mathrm{loc}^f_p)$.
Note that $\mathrm{H}^1_f(\mathbb{Q}_p, T)^\perp = \mathrm{H}^1_f(\mathbb{Q}_p, W^*(1))$.
Therefore, the proof of the finiteness of $\mathrm{Sel}( \mathbb{Q}, W^*(1) )$ reduces to
showing that the rational restriction map 
$$\mathrm{loc}^s_p : \mathrm{Sel}^{  \lbrace p \rbrace }( \mathbb{Q}, V ) \to \mathrm{H}^1_{/f}(\mathbb{Q}_p, V)$$
is surjective.
By Theorem \ref{thm:kato-finiteness-mordell-weil}, we have
\[
\xymatrix@R=0em{
\mathrm{Sel}^{  \lbrace p \rbrace }( \mathbb{Q}, V ) \ar[r]^-{\mathrm{loc}^s_p} & \mathrm{H}^1_{/f}(\mathbb{Q}_p, V) \ar[r]^-{\mathrm{exp}^*}_-{\simeq} & \mathrm{Fil}^0(\mathbf{D}_{\mathrm{cris}}(V)) \\
z_{\mathbb{Q}} \ar@{|->}[rr] & & \mathrm{exp}^* \circ \mathrm{loc}^s_p ( z_{\mathbb{Q}} ) = \dfrac{L^{(p)}(E,1)}{\Omega^+_E} \cdot \omega_E \neq 0 ,
}
\]
so $\mathrm{loc}^s_p$ is surjective. Theorem \ref{thm:kato-finiteness-mordell-weil-sha} follows.

\subsection{The Kolyvagin system argument}
We move to \emph{Kolyvagin systems}.
\subsubsection{}
First, why Kolyvagin systems? As we mentioned before, there are at least two advantages.
\begin{enumerate}
\item The sharp bound via the primitivity, which gives a mod $p$ criterion for verifying the exact Bloch--Kato type formula and the main conjecture.
\item The structure theorem, not only the size (e.g. \cite{kolyvagin-selmer}).
\end{enumerate}
We focus on the second advantage in this article.
By utilizing the Kolyvagin system argument, the following statement (Theorem \ref{thm:finiteness-rubin})
\begin{quote}
If $\mathbf{z} = \lbrace z_{K} \rbrace_K$ is an Euler system and $z_{\mathbb{Q}}$ is not a torsion, then
$\mathrm{Sel}_{ \lbrace p \rbrace }( \mathbb{Q}, W^*(1) )$ is finite.
\end{quote}
can be refined as follows:
\begin{quote}
Let $\ks = \lbrace \kappa_n \rbrace_n$ be the Kolyvagin system attached to the above Euler system $\mathbf{z}$.
If $\kappa_1 = z_{\mathbb{Q}}$ is not a torsion, then
the structure of $\mathrm{Sel}_{ \lbrace p \rbrace }( \mathbb{Q}, W^*(1) )$
is described in terms of all $\kappa_n$'s.
\end{quote}
What's the cost for this refinement?
Let $z_{K}$ be an Euler system class and $\Sigma'$ be the finite set of primes $\ell$ where $\ell$ is ramified in $K/\mathbb{Q}$.
Applying the Kolyvagin derivative $D_K$ (e.g. $\S$\ref{subsubsec:kolyvagin-system-argument}) to an Euler system class $z_{K}$, we have \emph{a priori} 
$$\kappa_{\Sigma'} := D_K z_K \pmod{\varpi^m} \in \mathrm{Sel}^{\lbrace p \rbrace \cup \Sigma'}(\mathbb{Q}, W_m) .$$
Following the theory of Kolyvagin systems, 
we can organize $\kappa_{\Sigma'}$ with more controlled local conditions at primes in $\Sigma'$.
Let $\ell$ be a prime with $\ell \equiv 1 \pmod{\varpi^m}$.
Then the transverse local condition at $\ell$ is defined by
$$\mathrm{H}^1_{\mathrm{tr}}(\mathbb{Q}_\ell, W_m) = \mathrm{ker} \left(
\mathrm{H}^1(\mathbb{Q}_\ell, W_m) \to \mathrm{H}^1(\mathbb{Q}_\ell(\zeta_\ell), W_m)
 \right),$$
and it can be viewed as the complement of the finite local condition $\mathrm{H}^1_{f}$ at $\ell$.
There exists the finite-singular comparison map
$$\varphi^{\mathrm{fs}}_{\ell} : \mathrm{H}^1_f(\mathbb{Q}_\ell, W_m) \to \mathrm{H}^1_{/f}(\mathbb{Q}_\ell, W_m) = \mathrm{H}^1_{\mathrm{tr}}(\mathbb{Q}_\ell, W_m)$$
obtained from the Euler factor at $\ell$.
Let $n = \prod_{\ell \in \Sigma'} \ell$, and we would like to construct
$$\kappa_n \in \mathrm{Sel}^{\lbrace p \rbrace}_{n\textrm{-}\mathrm{tr}}(\mathbb{Q}, W_m),$$
i.e.
$\mathrm{loc}_\ell(\kappa_n) \in \mathrm{H}^1_{\mathrm{tr}}(\mathbb{Q}_\ell, W_m)$
for every $\ell \in \Sigma'$.
Following \cite[Appendix A]{mazur-rubin-book}, there exists an explicit (but complicated) formula for
$\kappa_n$ in terms of $\kappa_{\Sigma''}$'s where $\Sigma''$ runs over subsets of $\Sigma'$.
The axiom for Kolyvagin systems is the following local relation
$$\mathrm{loc}_\ell (\kappa_{n\ell}) =  \varphi^{\mathrm{fs}}_{\ell} ( \mathrm{loc}_\ell(\kappa_n) ) .$$
\begin{prop}
\begin{enumerate}
\item $\mathrm{H}^1_f(\mathbb{Q}_\ell, W_m)$ and $\mathrm{H}^1_f(\mathbb{Q}_\ell, W^*_m(1))$ are orthogonal to each other with respect to the local Tate pairing $\langle - , - \rangle_{\ell}$. 
\item $\mathrm{H}^1_{\mathrm{tr}}(\mathbb{Q}_\ell, W_m)$ and $\mathrm{H}^1_{\mathrm{tr}}(\mathbb{Q}_\ell, W^*_m(1))$ are orthogonal to each other with respect to the local Tate pairing $\langle - , - \rangle_{\ell}$.
\end{enumerate}
\end{prop}
\begin{proof}
See \cite[Prop. 1.3.2]{mazur-rubin-book}.
\end{proof}
\begin{rem}
\begin{enumerate}
\item We consider the core rank one case only. There are the theories of higher (core) rank Kolyvagin systems and Stark systems  \cite{mazur-rubin-control,sakamoto-stark-systems,burns-sano-imrn,burns-sakamoto-sano-2,burns-sakamoto-sano-3}.
\item The image of the Galois representation is expected to be sufficiently large as described in $\S$\ref{subsubsec:image}.
\end{enumerate}
\end{rem}
\subsubsection{} \label{subsubsec:kolyvagin-system-argument}
From now on, we restrict ourselves to the case of elliptic curves again.
Let $\mathbf{z} = \left\lbrace z_K \right\rbrace_{K}$ be Kato's Euler system.
Then
$z_K \in \mathrm{H}^1(K,T)$ is characterized by
$$\sum_{\sigma \in \mathrm{Gal}(K/\mathbb{Q})} \left( \mathrm{exp}^* \circ \mathrm{loc}^s_p z^{\sigma}_K \right) \cdot \chi(\sigma)
= \dfrac{L^{(Sp)}(E,\chi, 1)}{\Omega^{\chi(-1)}_E} \cdot \omega_E$$
where
\begin{itemize}
\item $\chi$ is an even character of $\mathrm{Gal}(K/\mathbb{Q})$, and
\item $S$ is the product of the ramified primes of $K/\mathbb{Q}$.
\end{itemize}
For an integer $m \geq 1$, denote by $\mathcal{P}_m$ the set of primes $\ell$ such that
$(\ell, Np) = 1$, $a_\ell \equiv \ell +1 \pmod{p^m}$ and $\ell \equiv 1 \pmod{p^m}$.
For each $\ell \in \mathcal{P}_m$, fix a primitive root $\eta_{\ell}$ mod $\ell$ and
write
\[
\xymatrix@R=0em{
\mathrm{Gal}(\mathbb{Q}(\zeta_\ell)/\mathbb{Q}) & (\mathbb{Z}/\ell\mathbb{Z})^\times \ar[l]_-{\simeq} \\
\sigma_{\eta_\ell} & \eta_\ell \ar@{|->}[l] .
}
\]
the Kolyvagin derivative operator $D_\ell$ at $\ell \in \mathcal{P}_m$  is defined by
$$D_\ell = \sum i \cdot \sigma^i_{\eta_\ell}$$
Let $\mathcal{N}_m$ be the set of square-free products of primes in $\mathcal{P}_m$ and  $n \in \mathcal{N}_m$.
Then we define
$D_n = \prod_{\ell \vert n} D_\ell$.
Write $G_n = \mathrm{Gal}(\mathbb{Q}(\zeta_n)/\mathbb{Q})$ for convenience.
Let $\mathrm{ES}(\mathbb{Q}(\zeta_n), T) \subseteq \mathrm{H}^1(\mathbb{Q}(\zeta_n), T)$ be the $\mathbb{Z}_p\llbracket\mathrm{Gal}(\mathbb{Q}(\zeta_n)/\mathbb{Q})\rrbracket$-submodule generated by 
cohomology classes in $\mathrm{H}^1(\mathbb{Q}(\zeta_n), T)$ which are parts of Euler systems for $(T, \mathcal{K})$. See also \cite[$\S$4.2]{rubin-book}.
Then we have the commutative diagram
\[
\xymatrix{
\mathrm{ES}(\mathbb{Q}(\zeta_n), T) \ar[r]^-{D_n} \ar[rrdd] & D_n\left( \mathrm{ES}(\mathbb{Q}(\zeta_n), T) \right) \ar[r]^-{\bmod{p^m}} & \left( \mathrm{H}^1(\mathbb{Q}(\zeta_n), T)/p^m \right)^{G_n} \ar@{_{(}->}[d]  \\
& & \mathrm{H}^1(\mathbb{Q}(\zeta_n), T/p^m T)^{G_n} \ar[d]_-{\simeq}^-{\mathrm{res}^{-1}} \\
& & \mathrm{H}^1(\mathbb{Q}, T/p^m T)
}
\]
and the Euler system class $z_{\mathbb{Q}(\zeta_n)}$ maps to the Kolyvagin system class $\kappa_n$ following the above diagram 
$$z_{\mathbb{Q}(\zeta_n)} \mapsto
D_n z_{\mathbb{Q}(\zeta_n)}
\mapsto
D_n z_{\mathbb{Q}(\zeta_n)} \pmod{p^m}
\mapsto
\kappa_n .$$
Here, the restriction map is an isomorphism since we assume the large Galois image assumption.
Even without this assumption, we are still able to find the canonical inverse image of a derivative class under the restriction map. See \cite[Rem. 4.4.3]{rubin-book} for details.
\begin{rem}
Although there is an explicit formula for $\kappa_n$ in terms of $D_{n'} z_{\mathbb{Q}(\zeta_{n'})} \pmod{p^m}$ for $n'$ dividing $n$,
we can just identify  $D_n z_{\mathbb{Q}(\zeta_n)} \pmod{p^m}$ with $\kappa_n$ for this case.
It is essentially because Euler factors are quadratic.
\end{rem}

\subsubsection{}
We explain how the Kolyvagin system argument works \cite[Prop. 4.5.8]{mazur-rubin-book}.
We now fix \emph{one} integer $m \geq 1$ while every positive integer $m$ was considered together before.
Let
$\ks^{(m)} = (\kappa^{(m)}_n)_{n\in \mathcal{N}_m}$ be the mod $p^m$ reduction of Kato's Kolyvagin system $\ks$,
and
\begin{align*}
 \lambda (n, E[p^m]) & = \mathrm{length}_{\mathbb{Z}_p} \left( \mathrm{Sel}_{\lbrace p \rbrace, n\textrm{-}\mathrm{tr}}(\mathbb{Q}, E[p^m]) \right), \\
\partial^{(r)}(\ks^{(m)}) & = \mathrm{min} \left\lbrace  m - \mathrm{length}_{\mathbb{Z}_p} \left( \mathbb{Z}/p^m\mathbb{Z} \cdot \kappa^{(m)}_n  \right) : n \in \mathcal{N}_m, \nu(n) = r\right\rbrace.
\end{align*}

\begin{thm}[Mazur--Rubin]
Suppose that $\ks^{(m)}$ is non-trivial.
Then there exists an integer $j \geq 0$  such that
$$\mathbb{Z}/p^m\mathbb{Z} \cdot \kappa^{(m)}_n = p^{j + \lambda(n, E[p^m])} \cdot \mathrm{Sel}^{\lbrace p \rbrace}_{ n\textrm{-}\mathrm{tr} }(\mathbb{Q}, E[p^m]) $$
for every $n \in \mathcal{N}_m$. 
\end{thm}
\begin{proof}
See \cite[Cor. 4.5.2]{mazur-rubin-book}.
\end{proof}
\begin{thm}[Mazur--Rubin]
Suppose that $\ks^{(m)}$ is non-trivial.
Write 
$$\mathrm{Sel}_{\lbrace p \rbrace}( \mathbb{Q}, E[p^m]) \simeq \bigoplus_{i \geq 1} \mathbb{Z}/p^{d_i} \mathbb{Z}$$
with $d_1 \geq d_2 \geq \cdots $.
Then for every $r \geq 0$, we have 
$$\partial^{(r)}(\ks^{(m)}) = \mathrm{min} \left\lbrace m, j + \sum_{i > r }d_i \right\rbrace . $$
\end{thm}
\begin{proof}
This is easy to get the equality when $r = 0$.
When $r > 0$, consider the map
$$\bigoplus_{\ell \vert n}  \mathrm{loc}_\ell : \mathrm{Sel}_{\lbrace p \rbrace}(\mathbb{Q}, E[p^m]) \to \bigoplus_{\ell \vert n}  E(\mathbb{Q}_\ell) \otimes \mathbb{Z}/p^m\mathbb{Z}\simeq (\mathbb{Z}/p^m\mathbb{Z})^{\oplus \nu(n)} .$$
Then the kernel of this map is contained in
$$\mathrm{Sel}_{\lbrace p \rbrace,  n\textrm{-}\mathrm{tr}} (\mathbb{Q}, E[p^m]) .$$
Therefore, $\geq$ is done, so it suffices to prove the opposite inequality.
Using Chebotarev density theorem, we can choose a prime $\ell \in \mathcal{P}_m$
such that
\begin{itemize}
\item $\mathrm{Sel}_{\lbrace p , \ell \rbrace}(\mathbb{Q}, E[p^m]) \simeq \bigoplus_{i \geq 2} \mathbb{Z}/p^{d_i} \mathbb{Z}$, and
\item $\mathrm{Sel}_{\lbrace p, \ell \rbrace}(\mathbb{Q}, E[p^m]) = \mathrm{Sel}_{\lbrace p \rbrace,  \ell\textrm{-}\mathrm{tr}}(\mathbb{Q}, E[p^m])$.
\end{itemize}
Then use induction on $\nu(n)$.
It is remarkable that this process kills a generator in each step.
\end{proof}

\section{Lecture 4: A sketch of the construction of Beilinson--Kato's zeta elements}
This last lecture on the construction may be the most painful part of this lecture series although we give a sketch only, give no proof, and do not check the norm compatibility.
\begin{itemize}
\item It would not be very important to you if you mainly concern the \emph{applications} of Kato's Euler systems.
\item It must be very important to you if you are interested in the construction of \emph{new} Euler systems based on Kato's construction (e.g. Loeffler--Zerbes' program).
\end{itemize}
When I first try to read \cite{kato-euler-systems}, I was quite confused with the notation
\[
\xymatrix{
{}_{c,d} z^{(p)}_m (f, r,r', \xi, S) , & \mathbf{z}^{(p)}_\gamma .
}
\]
Of course,  $f$ comes from the modular form and $m$ comes from $\mathbb{Q}(\zeta_m)$.
However, where do all the other ``variables" come from?
We need to understand the geometry of modular curves and specific functions on them.
\begin{caveat} \label{caveat:universal-elliptic-curve}
In this last lecture, the notation $E$ is used for the universal elliptic curve over modular curves, not for elliptic curves over $\mathbb{Q}$ mentioned in former lectures.
\end{caveat}
\subsection{Siegel units and zeta elements}
\subsubsection{Modular forms and elliptic curves} We follow \cite[$\S$1.4]{kato-euler-systems}.
Let $N \geq 3$ be an integer and $Y(N)$ the modular curve over $\mathbb{Q}$ of level $N$ without cusps representing the functor
\[
\xymatrix@R=0em{
\underline{\mathrm{Sch}}_{/\mathbb{Q}} \ar[r] & \underline{\mathrm{Sets}} \\
S  \ar@{|->}[r] & \textrm{the isomorphism classes of } (E, e_1, e_2)
}
\]
where 
\begin{itemize}
\item $\underline{\mathrm{Sch}}_{/\mathbb{Q}}$ is the category of schemes over $\mathbb{Q}$ and $\underline{\mathrm{Sets}}$ is the category of sets,  and
\item $E$ is an elliptic curve over $S$,  and $(e_1, e_2)$ is a pair of sections of $E/S$ and it forms a $\mathbb{Z}/N\mathbb{Z}$-basis of $E[N]$.
\end{itemize}
This is a smooth irreducible affine curve, and the constant field of $Y(N)$ is $\mathbb{Q}(\zeta_N)$, not $\mathbb{Q}$. 
Denote by $X(N) \supseteq Y(N)$ the smooth compactification of $Y(N)$.

Let $n$ be an integer with $(n,N) = 1$. By using the finite \'{e}tale surjective map
\[
\xymatrix@R=0em{
Y(Nn) \ar@{->>}[r] & Y(N) \\
(E, e_1, e_2) \ar@{|->}[r] & (E, n \cdot e_1, n \cdot e_2) ,
}
\]
we have inclusion 
$\mathcal{O}(Y(N)) \subseteq \mathcal{O}(Y(Nn))$
via the pull-back of this map.

\subsubsection{Theta functions}
The \emph{theta function} in the following proposition is the very starting point of the construction.
\begin{prop} \label{prop:theta-functions}
Let $E$ be an elliptic curve over a scheme $S$.
Let $c$ be an integer with $(c, 6)  =1$.
\begin{enumerate}
\item There exists a unique ${}_c\theta_E \in \mathcal{O}(E \setminus E[c])^\times$
such that
\begin{enumerate}
\item $\mathrm{Div}  ( {}_c\theta_E ) = c^2 \cdot (0) - E[c]$ on $E$
\item $\mathrm{N}_a({}_c\theta_E) = {}_c\theta_E$ for an integer $a$ with $(a, c) = 1$
where $\mathrm{N}_a$ is the norm map $\mathrm{N}_a : \mathcal{O}(E \setminus E[ac])^\times \to \mathcal{O}(E \setminus E[c])^\times$
associated to the pull-back homomorphism by the multiplication by $a$.
\end{enumerate}
\item Let $d$ be another integer with $(d,6) = 1.$ Then
$$({}_d\theta_E )^{c^2} \cdot (c^*( {}_d\theta_E ) )^{-1} =  ({}_c\theta_E )^{d^2} \cdot (d^* ( {}_c\theta_E ) )^{-1}$$
in $\mathcal{O}(E \setminus E[ac])^\times$
where the maps
\begin{align*}
c^* & : \mathcal{O}(E \setminus E[d])^\times \to \mathcal{O}(E \setminus E[cd])^\times, \\
d^* &  : \mathcal{O}(E \setminus E[c])^\times \to \mathcal{O}(E \setminus E[cd])^\times
\end{align*}
are the pull-back homomorphisms by the multiplication by $c$ and $d$, respectively.
\item Over $\mathbb{C}$,  ${}_c\theta_E$ can be written more explicitly.
Let $\mathfrak{h}$ be the upper-half plane.
Let $\tau \in \mathfrak{h}$,  $z \in \mathbb{C} \setminus c^{-1} (\mathbb{Z}\tau + \mathbb{Z})$.
Then ${}_c\theta(\tau, z)$ is the value at $z$ of ${}_c\theta_{\mathbb{C}/(\mathbb{Z}\tau + \mathbb{Z})}$ over $\mathbb{C}$.
Therefore, we have the explicit formula
$${}_c\theta(\tau, z) = q^{\frac{1}{12} \cdot (c^2-1)} \cdot (-t)^{\frac{1}{2} \cdot (c- c^2)} \cdot \gamma_q(t)^{c^2} \cdot  \gamma_q(t^c)^{-1}$$
where $q = e^{2 \pi i \tau}$, $t = e^{2 \pi i z}$, and
$$\gamma_q(t) = \prod_{n \geq 0} (1 - q^n \cdot t) \cdot \prod_{n \geq 1} (1 - q^n \cdot t^{-1}) .$$ 
\item If $h : E \to E'$ is an isogeny of elliptic curve of degree prime to $c$, then
we have the corresponding norm map
$h_*$ sends ${}_c\theta_E$ to ${}_c\theta_{E'}$.
\end{enumerate}
\end{prop}
\begin{proof}
See \cite[Prop. 1.3 and $\S$1.10]{kato-euler-systems} and \cite[$\S$1.2]{scholl-kato}.
\end{proof}

\subsubsection{Siegel units from theta functions}
Following \cite[$\S$1.4]{kato-euler-systems}, we review the construction of the Siegel units from the theta functions above
\[
\xymatrix{
{}_c g_{\alpha, \beta} \in \bigcup_{N} \mathcal{O}(Y(N))^\times , & g_{\alpha, \beta} \in \bigcup_{N} \mathcal{O}(Y(N))^\times \otimes \mathbb{Q}
}
\]
where
$(\alpha, \beta) \in (\mathbb{Q}/\mathbb{Z})^{\oplus 2} \setminus \lbrace (0,0) \rbrace$
and $(c, 6 \cdot \mathrm{ord}(\alpha) \cdot \mathrm{ord}(\beta)) = 1$.
Note that the former one is integral, but the latter allows denominators.
If $N \cdot \alpha = N \cdot \beta = 0$, then
\[
\xymatrix{
{}_c g_{\alpha, \beta} \in \mathcal{O}(Y(N))^\times , & g_{\alpha, \beta} \in \mathcal{O}(Y(N))^\times \otimes \mathbb{Q} .
}
\]
In this situation, we have relation
$${}_c g_{\alpha, \beta} = ( g_{\alpha, \beta} )^{c^2} \cdot ( g_{c \alpha, c\beta} )^{-1}
\in \mathcal{O}(Y(N))^\times \otimes \mathbb{Q} .$$
How to produce these elements?
Consider the universal elliptic curve
\[
\xymatrix{
E \ar[d]_-{\pi} \\
Y(N) \ar@/_1pc/[u]_-{e_1, e_2}
}
\]
where $\pi$ is the structure map and $N \geq 3$ is an integer satisfying $N \cdot \alpha = N \cdot \beta = 0$.
Write
$$( \alpha , \beta ) = (a/N  , b/N) \in (\frac{1}{N} \mathbb{Z}/\mathbb{Z})^{\oplus 2} \setminus \lbrace (0,0) \rbrace$$
with $a, b \in \mathbb{Z}$.
Consider the map
$$\iota_{\alpha, \beta} := a\cdot e_1 + b \cdot e_2 : Y(N) \to E \setminus E[c] \subseteq E .$$
Note that $\mathrm{im}(\iota_{\alpha, \beta}) \cap E[c] = \emptyset$
since $(c, \mathrm{ord}(\alpha) \cdot \mathrm{ord}(\beta)) = 1$ and $(\alpha, \beta) \neq (0,0)$.

Then we define
$${}_c g_{\alpha, \beta} := \iota^*_{\alpha, \beta} ( {}_{c}\theta_E ) \in \mathcal{O}(Y(N))^\times $$
and 
$$g_{\alpha, \beta} := {}_c g_{\alpha, \beta}  \otimes (c^2-1)^{-1}  \in \mathcal{O}(Y(N))^\times \otimes \mathbb{Q} $$
by taking an integer
$c$ satisfying 
\begin{itemize}
\item $(c,6) =1$,
\item $c \equiv 1 \pmod{N}$, and
\item $c \neq \pm 1$.
\end{itemize}
The resulting $g_{\alpha, \beta}$ is independent of such $c$ and is also independent of $N$.
\begin{rem}
The reader should check \cite[Rem. 1.5]{kato-euler-systems} for the important intuition on Siegel units and theta functions, which is not true logically. 
\end{rem}
\subsubsection{Group action on $Y(N)$}
We follow \cite[$\S$1.6]{kato-euler-systems}.
Consider the following (left) action $\mathrm{GL}_2(\mathbb{Z}/N\mathbb{Z})$ on $Y(N)$.
$$\sigma = 
\begin{pmatrix}
a & c \\ b & d
\end{pmatrix}  : (E, e_1, e_2) \mapsto (E, e'_1, e'_2)$$
where
$$\begin{pmatrix}
e'_1 \\
e'_2
\end{pmatrix}
=
\begin{pmatrix}
a & b \\ c & d
\end{pmatrix}
\cdot
\begin{pmatrix}
e_1 \\ e_2
\end{pmatrix}$$
and
$\sigma (\zeta_N) = \zeta^{\mathrm{det}(\sigma)}_N$
on the total constant field.
\subsubsection{A little bit of algebraic $K$-theory}
We recommend \cite{weibel-K-book} for the basic material on algebraic $K$-theory.

Let $X$ be a regular separate Noetherian scheme.
Then Quillen's $K$-groups $K_i(X)$ are defined for $i \geq 0$
with cup product 
$$\cup : K_i(X) \times K_j(X) \to K_{i+j}(X).$$
Two properties are important for our purpose.
\begin{enumerate}
\item there exists a canonical homomorphism
$$\mathcal{O}(X)^\times \to K_1(X).$$
\item there exists the universal (Steinberg) symbol map via cup product
$$\mathcal{O}(X)^\times \times \mathcal{O}(X)^\times \to K_2(X)$$
defined by sending a pair $(u, v)$ to the symbol $\lbrace u, v \rbrace$.
\end{enumerate}

\begin{rem}
\begin{enumerate}
\item When $X =\mathrm{Spec}(R)$ with commutative ring with unity $R$, we have
$$K_1(X) = K_1(R) = \mathrm{GL}(R) / [\mathrm{GL}(R), \mathrm{GL}(R)]$$
where 
$\mathrm{GL}(R) = \bigcup_{n \geq 1} \mathrm{GL}_n(R)$
under $\mathrm{GL}_n(R) \hookrightarrow \mathrm{GL}_{n+1}(R)$ defined by
$A \mapsto \begin{pmatrix}
A & 0  \\
0 & 1 
\end{pmatrix}$. In this situation, the first map is exactly
$$\mathcal{O}(X)^\times = \mathrm{GL}_1(R) \to \mathrm{GL}(R) \to K_1(X).$$
\item If $X = \mathrm{Spec}(k)$ for a field $k$, then
$$K_2(X) = K_2(k) = ( \wedge^2 k^\times ) / ( \lbrace u, 1-u \rbrace : u \in k^\times ) .$$
\end{enumerate}

\end{rem}

\subsubsection{More modular curves} \label{subsubsec:more-modular-curves}
We follow \cite[$\S$2.1]{kato-euler-systems}.
Let $M, N \geq 1$ be integers, and we now define $Y(M,N)$.
Choose a sufficiently big $L \geq 3$ such that $M$ divides $L$ and $N$ divides $L$.
Then we define
$$Y(M,N) = G \backslash Y(L), $$
which is independent of $L$,
where
$$G = \left\lbrace 
\begin{pmatrix}
 a & b \\ c & d
\end{pmatrix} 
\in \mathrm{GL}_2(\mathbb{Z}/L\mathbb{Z}) : \gamma \equiv 
\begin{pmatrix}
 1 & 0 \\ * & *
\end{pmatrix}
\pmod{M},
 \gamma \equiv 
\begin{pmatrix}
 * & * \\ 0 & 1
\end{pmatrix}
\pmod{N}
\right\rbrace .$$
Let $X(M,N) \supseteq Y(M,N)$ be the smooth compactification of $Y(M,N)$.
Suppose that $M+N \geq 5$. 
Then $Y(M,N)$ represents the functor
\[
\xymatrix@R=0em{
\underline{\mathrm{Sch}}_{/\mathbb{Q}} \ar[r] & \underline{\mathrm{Sets}} \\
S  \ar@{|->}[r] & \textrm{the isomorphism classes of } (E, e_1, e_2)
}
\]
where 
\begin{itemize}
\item $E$ is an elliptic curve over $S$,  and
\item $(e_1, e_2)$ is a pair of sections of $E/S$ such that $M \cdot e_1 = N \cdot e_2 = 0$ and the map $\mathbb{Z}/M\mathbb{Z} \times \mathbb{Z}/N\mathbb{Z} \to E$ defined by sending $(a,b) \mapsto a \cdot e_1 + b  \cdot e_2$ is injective.
\end{itemize}
Since $M$ divides $L$ and $N$ divides $L$, we have natural map
\[
\xymatrix@R=0em{
Y(L) \ar[r] & Y(M,N) \\
(E, e_1, e_2)  \ar@{|->}[r] & (E, (L/M)\cdot e_1, (L/N)\cdot e_2) .
}
\]
\subsubsection{Zeta elements in $K_2$, zeta functions, and Beilinson's theorem}
We follow \cite[$\S$2.2]{kato-euler-systems}.
We produce one zeta element in $K_2$ of $Y(M,N)$ from two Siegel units as follows.

Let $M, N \geq 2$ be integers such that $M+N \geq 5$.
Choose $c, d \in \mathbb{Z}_{\geq 1}$ such that $(c,6M) = 1$ and $(d,6N) = 1$. 
 The zeta elements are defined by
 \begin{align*}
 {}_{c,d} z_{M,N} & := \lbrace {}_c g_{1/M, 0}, {}_d g_{0, 1/N} \rbrace  \in K_2(Y(M,N)) , \\
z_{M,N} & := \lbrace  g_{1/M, 0},  g_{0, 1/N} \rbrace  \in K_2(Y(M,N)) \otimes \mathbb{Q} 
 \end{align*}
satisfying certain norm compatibilities \cite[Props. 2.3 and 2.4]{kato-euler-systems}.

Then we can define the corresponding zeta function  \cite[$\S$2.5]{kato-euler-systems}, which is the Hecke operator version of the zeta function associated to the zeta element $z_{M,N}$
$$Z_{M,N}(s) : \lbrace s \in \mathbb{C} : \mathrm{Re}(s) > 2 \rbrace \to \mathrm{End}_{\mathbb{C}}( \mathrm{H}^1(Y(M,N)(\mathbb{C}), \mathbb{C}) )$$
defined by
$$Z_{M,N}(s)= \sum_{n \geq 1, (n,M)=1} T'(n) \cdot \begin{pmatrix}
1/n & 0 \\ 0 & 1
\end{pmatrix}^* \cdot n^{-s}$$
where $T'(n)$ is the dual Hecke operator acting on $\mathrm{H}^1(Y(M,N)(\mathbb{C}), \mathbb{C})$ and 
$\begin{pmatrix}
1/n & 0 \\ 0 & 1
\end{pmatrix}^*$ is the pull-back by the action of $\begin{pmatrix}
1/n & 0 \\ 0 & 1
\end{pmatrix}$.

Recall the isomorphism coming from Poincar\'{e} duality
\[
\xymatrix@R=0em{
\mathrm{H}_1(X(M,N)(\mathbb{C}), \lbrace\mathrm{cusps}\rbrace, \mathbb{Z}) \otimes \mathbb{C} \ar[r] & \mathrm{H}^1(Y(M,N)(\mathbb{C}), \mathbb{C}) \\
(0, \infty) \ar@{|->}[r] & \delta_{M,N} 
}
\]
where the homology class $(0, \infty) \in \mathrm{H}_1(X(M,N)(\mathbb{C}), \lbrace\mathrm{cusps}\rbrace, \mathbb{Z})$
is represented by the geodesic from 0 to $\infty$ on the upper half plane.
See \cite[$\S$2.7]{kato-euler-systems} for details.
\begin{thm}[Beilinson] \label{thm:beilinson}
There exists the regulator map
$$\mathrm{reg}_{M,N} : K_2(Y(M,N)) \to \mathrm{H}^1(Y(M,N)(\mathbb{C}), \mathbb{R}\cdot i)$$
sending $z_{M,N}$
to
$$\lim_{s \to 0} \dfrac{1}{s} \cdot Z_{M,N}(s) \cdot 2 \pi i \cdot \delta^-_{M,N}$$
\end{thm}
\begin{proof}
See \cite[Thm. 2.6 and $\S$7]{kato-euler-systems} and \cite[$\S$5]{beilinson-higher-regulators}.
\end{proof}
\begin{rem}
Beilinson's theorem relates $z_{M,N}$ to \emph{non-critical} values of $L$-functions.
In Kato's construction, he used the Chern character map and the Soul\'{e} twist to $z_{M,N}$ to obtain the arithmetic information of all the \emph{critical} $L$-values.
\end{rem}
Up to now, we concern the elements on $Y(M,N)$ over $\mathbb{Q}$.
However, in order to have Euler systems in the sense of the third lecture, we need to construct certain elements on $Y_1(N)$ over $\mathbb{Q}(\zeta_m)$.

\subsection{Euler systems on $X_1(N)$ over $\mathbb{Q}(\zeta_m)$: $(\xi, S)$} \label{subsec:changing-modular-curves}
We briefly summarize how the elements in $K_2(Y(M,N))$ map to the elements in $K_2(Y_1(N) \otimes \mathbb{Q}(\zeta_m))$.
There are two types of zeta elements in terms of $(\xi, S)$:
\begin{enumerate}
\item $\xi = a(A)$ and $S = \lbrace \ell : \ell \vert mA \rbrace$ 
where $a(A)$ is just a symbol with $a \in \mathbb{Z}$ and $A \in \mathbb{Z}_{\geq 1}$.
\item $\xi \in \mathrm{SL}_2(\mathbb{Z})$ and $S = \lbrace \ell : \ell \vert mN \rbrace$.
\end{enumerate}
\begin{rem}
The zeta elements coming from (1) recover the bad Euler factors, and the zeta elements coming from (2) are useful for the integrality. 
\end{rem}
From $(\xi, S)$, we sketch the construction of the zeta element (the ``without $(c,d)$" version)
$$z_{1, N,m} (\xi, S) \in K_2(Y_1(N) \otimes \mathbb{Q}(\zeta_m)) \otimes \mathbb{Q}$$
as follows.
\subsubsection{The first case}
Choose $M \geq 1$, $L \geq 4$ such that $mA \vert M$, $N \vert L$, and $M \vert L$.
Let
$S = \lbrace \ell : \ell \vert M \rbrace$, and
$L = \lbrace \ell : \ell \in S \textrm{ or } \ell \vert N \rbrace$.
Consider the map on the upper-half plane
$\mathfrak{h} \to \mathfrak{h} $ defined by
sending $\tau$ to $A^{-1}(\tau +a)$.
This map is compatible with the following natural map
$$Y(M,L) \to Y_1(N) \otimes \mathbb{Q}(\zeta_m) .$$
The corresponding norm map on $K$-groups 
$$t_{m, a(A)} : K_2(Y(M,L)) \otimes \mathbb{Q} \to K_2(Y_1(N) \otimes \mathbb{Q}(\zeta_m)) \otimes \mathbb{Q}$$
is defined and it sends $z_{M,L}$ to $z_{1, N,m}(a(A), S)$.
Furthermore, $z_{1, N,m}(a(A), S)$ is independent of $M$ and $L$.

\subsubsection{The second case}
Choose $L \geq 3$ such that $m \vert L$, $N \vert L$, and let $S = \lbrace \ell : \ell \vert L \rbrace$.
Consider the canonical projection
$Y(L) \to Y_1(N) \otimes \mathbb{Q}(\zeta_m)$,
and the corresponding norm map on $K$-groups
$$K_2(Y(L)) \otimes \mathbb{Q} \to K_2(Y_1(N)\otimes \mathbb{Q}(\zeta_m)) \otimes \mathbb{Q}$$
is defined and it sends $\xi^*(z_{L,L})$ to $z_{1,N,m}(\xi, S)$.
As before, $z_{1,N,m}(\xi, S)$ is also independent of $L$.

\subsubsection{How about their integral versions?}
There exist the ``with $(c,d)$" versions of the above construction and these are integral, so $\otimes \mathbb{Q}$ can be removed.

\subsection{A summary of the construction and so what?}
\subsubsection{}
We recall some parts of $\S$\ref{subsec:rough-picture} elementwisely.
\[
\xymatrix{
{}_c\theta_E \ar@{|->}[r] \ar@{|->}[d]_-{``\mathrm{dlog}"} & \textrm{Siegel units} \ar@{|->}[r]^-{\cup} & { \substack{ \textrm{cup product of } \\ \textrm{two Siegel units in} \\  K_2(Y(M,N)) } }  \ar@{|->}[d]^-{(*)} \\
\textrm{Eisenstein series} \ar@{|->}[dr]_-{\textrm{product}} & & { \substack{ \textrm{zeta elements in} \\  \mathrm{H}^1(\mathbb{Q}, \mathrm{H}^1_{\et}(Y(M,N), \mathbb{Z}_p)(k-r)) } }  \ar@{|->}[dl]^-{ {\substack{\textrm{dual exponential} \\ \textrm{\cite[Thm. 9.7]{kato-euler-systems}}} } } \\
&  { \substack{ \textrm{zeta modular forms} \\  = \ \textrm{product of  two Eisenstein series} } }  \ar@{|->}[d]_-{\textrm{Rankin--Selberg}} \\
& \textrm{zeta value formula} 
}
\]
where
\begin{itemize}
\item  $``\mathrm{dlog}"$ is an explicit formula involving logarithmic derivatives as in (\ref{eqn:eisenstein-series-siegel-units}) (see \cite[\S3.2]{kato-euler-systems} for details), and
\item  $(*)$ involves the Chern character map $\mathrm{Ch}_{M,N}(k, r, r')$ \cite[(8.4.3)]{kato-euler-systems},
Hochschild--Serre spectral sequence, and the Soul\'{e} twist. The parameters $(r, r')$ are chosen in this process.
\end{itemize}
In order to have a genuine Euler system, we need to rewrite everything with $Y(M,N)$ by that with $Y_1(N) \otimes \mathbb{Q}(\zeta_m)$.
The projection to an eigenform $f$ can be taken by the quotient by the height one prime ideal associated to $f$ of the Hecke algebra and it yields the zeta modular form for $f$.

\subsubsection{}
We are now able to explain the relations among ${}_{c,d} z^{(p)}_m (f, r,r', \xi, S)$, $\mathbf{z}^{(p)}_\gamma$, and $z_{\mathbb{Q}(\zeta_n)}$.

In the element ${}_{c,d} z^{(p)}_m (f, r,r', \xi, S)$, the parameters come from the following choices.
\begin{itemize}
\item the choice of $(c, d)$ comes from the choice of two theta functions (Proposition \ref{prop:theta-functions}),
\item the choice of $(r, r')$ comes from the Chern character map $\mathrm{Ch}_{M,N}(k,r,r')$ in \cite[(8.4.3)]{kato-euler-systems}, and
\item the choice of $(\xi, S)$ comes from the maps from  $Y(M,N)$  to  $Y_1(N) \otimes \mathbb{Q}(\zeta_m)$ in $\S$\ref{subsec:changing-modular-curves}.
\end{itemize}
The element $\mathbf{z}^{(p)}_\gamma$ is defined over $\mathbb{Q}(\zeta_{p^\infty})$ and is explicitly built out from
$${}_{c,d} \mathbf{z}^{(p)}_{p^n} (f, r,r', \xi, S) = \varprojlim_n {}_{c,d} z^{(p)}_{p^n} (f, r,r', \xi, S) .$$
See \cite[$\S$13.9]{kato-euler-systems} for the explicit formula.
Here, $\gamma$ is an element of $V_f = V_pE(-1)$.
We now suppose that the image of $\rho$ contains the conjugate of $\mathrm{SL}_2(\mathbb{Z}_p)$.
Then $\mathrm{H}^1_{\mathcal{I}w}(\mathbb{Q}, T_f) = \varprojlim_{n} \mathrm{H}^1(\mathbb{Q}(\zeta_{p^n}), T_f)$
is free of rank one over $\mathbb{Z}_p\llbracket \mathrm{Gal}(\mathbb{Q}(\zeta_{p^\infty})/\mathbb{Q}) \rrbracket$.
Note that $\mathrm{H}^1_{\mathrm{Iw}}$ in \S\ref{subsubsec:iwasawa-cohomology} is slightly different from $\mathrm{H}^1_{\mathcal{I}w}$.
Furthermore, $\varprojlim_{n} \mathrm{H}^1(\mathbb{Q}(\zeta_{mp^n}), T_f)$
is still free over $\mathbb{Z}_p\llbracket \mathrm{Gal}(\mathbb{Q}(\zeta_{p^\infty})/\mathbb{Q}) \rrbracket$ for $m \geq 1$.
By using the argument of \cite[$\S$13.14]{kato-euler-systems}, we obtain an \emph{integral} Euler system and this construction with a suitable choice of $\gamma \in T_f$ gives us $z_{\mathbb{Q}(\zeta_n)}$ in the second lecture. See \cite[Appendix A]{kim-nakamura} for the detail. Note that the choice of $\gamma$ is equivalent to the choice of periods.

\subsection{Eisenstein series}
We follow \cite[\S3]{kato-euler-systems}.
The motto is 
\[
``\mathrm{dlog} (\textrm{Siegel units}) = \textrm{Eisenstein series}."
\]
Remind that $E$ means the universal elliptic curve over the modular curve in this last lecture (Caveat \ref{caveat:universal-elliptic-curve}).
\subsubsection{Setup}
Consider the diagram
\[
\xymatrix{
E \ar[r] \ar[d]^-{\pi} & \overline{E} \ar[d]^-{\pi} \\
Y(N) \ar[r] & X(N)
}
\]
where $\overline{E}$ is smooth N\'{e}ron model of $E$ and $\pi$ in the right extends that in the left. 
Consider an invertible $\mathcal{O}_{X(N)}$-module
$$\mathrm{coLie}(\overline{E}) = \pi_*( \Omega^1_{\overline{E}/X(N)} ) (``=\omega") .$$
Then
the spaces of modular forms/cuspforms have the following geometric descriptions, respectively:
\begin{align*}
M_k(X(N)) & = \Gamma(X(N), \mathrm{coLie}(\overline{E})^{\otimes k} ) \\
& = \Gamma(X(N), \mathrm{coLie}(\overline{E})^{\otimes k-2} \otimes_{\mathcal{O}(X(N))} \Omega^1_{X(N)/\mathbb{Q}}(\mathrm{log}(\mathrm{cusps})))
\end{align*}
where $\Omega^1_{X(N)/\mathbb{Q}}(\mathrm{log}(\mathrm{cusps}))$ is the space of differential 1 forms with logarithmic poles at cusps,
and
\[
S_k(X(N)) = \Gamma(X(N), \mathrm{coLie}(\overline{E})^{\otimes k-2} \otimes_{\mathcal{O}(X(N))} \Omega^1_{X(N)/\mathbb{Q}}) .
\]
For a positive integer $c$ with $(c, 6N)=1$,
we define Eisenstein series ${}_c E^{(k)}_{\alpha, \beta} \in M_k(X(N))$ where
$k \geq 1$ is an integer and $(\alpha, \beta) \in (\mathbb{Z}/N\mathbb{Z})^{\oplus 2}$.

We now construct the product of Eisenstein series explicitly.
Let $r \in \mathbb{Z}$.
Consider the map
$$D : \pi^*\mathrm{coLie}(E)^{\otimes r} \to \pi^*\mathrm{coLie}(E)^{\otimes r+1}$$
defined by sending
$f \otimes \omega^{\otimes r}$ to 
$\dfrac{df}{\omega} \otimes \omega^{\otimes r+1}$
where
$\mathrm{coLie}(E) = \pi_*( \Omega^1_{E/Y(N)} )$
$\omega \in \mathrm{coLie}(E)$ is a local basis
$f \in \mathcal{O}_E$, and
$\dfrac{df}{\omega} \in \mathcal{O}_E$ is determined by
$$df = \dfrac{df}{\omega} \cdot \omega \in \Omega^1_{E/Y(N)}.$$

Since
\begin{align*}
\mathrm{dlog} ( {}_c \theta_E ) & \in \Gamma(E \setminus E[c], \Omega^1_{E/Y(N)} ) \\
& = \Gamma(E \setminus E[c], \pi^*\mathrm{coLie}(E) ),
\end{align*}
we have
$$D^k( \mathrm{dlog} ( {}_c \theta_E ) ) \in \Gamma(E \setminus E[c], \pi^*\mathrm{coLie}(E)^{\otimes k} ) .$$
Let $(\alpha , \beta) \in (1/N \mathbb{Z} / \mathbb{Z})^{\oplus 2} \setminus \lbrace (0,0) \rbrace$, and $(c, 6N) = 1$.
Then we have Eisenstein series
\begin{equation} \label{eqn:eisenstein-series-siegel-units}
{}_c E^{(k)}_{\alpha, \beta} = \iota^*_{\alpha, \beta} \left( D^{k-1}  \mathrm{dlog} ( {}_c \theta_E )  \right) \in \Gamma(Y(N), \mathrm{coLie}(E)^{\otimes k} )
\end{equation}
where $\iota_{\alpha, \beta} = a \cdot e_1 + b \cdot e_2 : Y(N) \to E \setminus E[c]$ with
\[
\xymatrix{
E \ar[d] \\
Y(N) \ar@/_1pc/[u]_-{e_1, e_2}
}
\]
and $(\alpha, \beta) = (a/N, b/N)$.
In fact, ${}_c E^{(k)}_{\alpha, \beta} \in M_k(X(N))$.
\subsubsection{}
What if $(\alpha, \beta) = (0,0)$?
We define
$${}_c E^{(k)}_{0, 0} = (a^k - 1)^{-1} \cdot \sum'_{\alpha, \beta}  {}_c E^{(k)}_{\alpha, \beta} $$
where $\sum^{'}$ is the sum runs over all non-zero elements $(\alpha, \beta)$ of $(\frac{1}{a}\mathbb{Z}/\mathbb{Z})^{\oplus 2}$
Then ${}_c E^{(k)}_{0, 0}$ is independent of $a$ if $a \neq \pm 1$ and is prime to $c$.

\subsubsection{}
When $k=2$?
There exists a variant
$${}_c \widetilde{E}^{(2)}_{\alpha, \beta} \in M_2(X(N))$$
which can be described in terms of an algebraic version  of Weierstrass $\wp$-function.

\subsubsection{}
There is another Eisenstein series (without $(c,d)$-version)
$$F^{(k)}_{\alpha, \beta} 
=
\left\lbrace
\begin{array}{ll}
 N^{-k} \cdot \sum_{x,y \in \mathbb{Z}/N\mathbb{Z}} E^{(k)}_{x/N, y/N} \cdot \zeta^{bx-ay}_N & \textrm{ if } k \neq 2 , \\
 N^{-2} \cdot \sum_{x,y \in \mathbb{Z}/N\mathbb{Z}} \widetilde{E}^{(k)}_{x/N, y/N} \cdot \zeta^{bx-ay}_N & \textrm{ if } k=2 \textrm{ and } (\alpha, \beta)\neq (0,0).
\end{array} \right.
$$
\subsubsection{}
Let
\begin{itemize}
\item $1 \leq r, r' \leq k-1$ and one of them is $k-1$,
\item if $r = k-2$ and $r' = k-1$, then $M \geq 2$, and
\item $(r,r') \neq (2, k-1)$, $(k-1,2)$, or $(k-1, k-2)$.
\end{itemize}
We are now able to define \textbf{zeta modular forms}
$$z_{M,N} (k, r,r') \in M_k(X(M,N))$$
by 
$$z_{M,N} (k, r,r')
:=
\left\lbrace
\begin{array}{ll}
 \textrm{(some constant)} \cdot F^{(k-r)}_{1/M, 0} \cdot E^{(r)}_{0, 1/N} & \textrm{ if } r' = k-1, \\
 \textrm{(some constant)} \cdot E^{(k-r')}_{1/M, 0} \cdot E^{(r')}_{0, 1/N} & \textrm{ if } r = k-1.
\end{array} \right.$$
The explicit formulas for zeta modular forms can be found in \cite[Prop. 5.8]{kato-euler-systems}.
The notion of zeta modular forms is the starting point for the proof of the zeta value formula. See \cite[$\S$7]{kato-euler-systems}.

\section*{Brief remarks on further developments}
We conclude this article by mentioning two recent developments of Kato's Euler systems very briefly.
\subsection*{Structural refinements}
As we mentioned in the third lecture, the language of Kolyvagin systems provides Mazur--Rubin's structure theorem for $p$-strict Selmer groups as well as a mod $p$ criterion of the Iwasawa main conjecture.
Here, we briefly explain how to extend Mazur--Rubin's structure theorem to Kurihara's structure theorem for standard Selmer groups \cite{kurihara-iwasawa-2012,kurihara-munster,kurihara-analytic-quantities}.

Recall that
$\ell \in \mathcal{P}_k$ means $\ell \equiv 1 \pmod{p^k}$ and $a_\ell \equiv \ell +1 \pmod{p^k}$.
and $n \in \mathcal{N}_k$ is a square-free product of the primes in $\mathcal{P}_k$.
and for $n \in \mathcal{N}_k$, let
$I_n = \prod_{\ell \vert n} (\ell - 1, 1- a_\ell + \ell) \subseteq  p^k\mathbb{Z}_p$.
The \textbf{Kurihara number at $n$} is defined by
$$\widedelta_n = \sum_{a \in (\mathbb{Z}/n\mathbb{Z})^\times} \overline{\left[\dfrac{a}{n} \right]^+} \cdot \prod_{\ell \vert n} \overline{\mathrm{log}_{\eta_\ell}(a)} \in \mathbb{Z}_p/I_n\mathbb{Z}_p.$$
The collection of Kurihara numbers can be viewed as the ``Betti realization" of Kato's \emph{Kolyvagin} systems, and it is also an elliptic curve analogue of higher annihilators in \cite{aoki-ideal-class-groups}.
Under the non-triviality of the collection of Kurihara numbers, we are able to describe the structure of $\mathrm{Sel}(\mathbb{Q}, E[p^\infty])$ in terms of $\widedelta_n$'s. 
Also, the non-triviality of Kurihara numbers is directly related to Kato's main conjecture.
These results sometimes give us more arithmetic information than the Birch and Swinnerton-Dyer conjecture.
See \cite{sakamoto-p-selmer, kim-structure-selmer} for details.

\subsection*{Variation in families}
Beilinson--Kato's zeta elements live in the \'{e}tale cohomology of modular curves (before $\mathrm{exp}^*$) or the space of modular forms (after $\mathrm{exp}^*$).
In order to deform Beilinson--Kato's zeta elements in $p$-adic families, we evidently need to put the cohomology of modular curves and the space of modular forms into their $p$-adic families first.

There are a variety of works at different levels of generalities under different technical assumptions. We list some of them.
\begin{enumerate}
\item Hida families (slope zero): \cite{ochiai-deformation,delbourgo-book,fukaya-kato-sharificonj}.
\item Coleman families (finite slope): \cite{kazim-kato-in-families,wang-kato-in-families-1,wang-kato-in-families-2,benois-buyukboduk-beilinson-kato,benois-buyukboduk-exceptional-zeros,benois-buyukboduk-critical,hansen-critical-one}.
\item Deformation rings (no slope condition): \cite{nakamura-kato-deformation,colmez-wang}.
\end{enumerate}
The deformation of Beilinson--Kato's zeta elements essentially generalizes
the ``zeta morphism" (not the zeta element!)
$$T_f \to \mathrm{H}^1_{\mathcal{I}w}(\mathbb{Q}, T_f)$$
 sending $\gamma$ to $\mathbf{z}^{(p)}_\gamma$.
More precisely, the Galois representation $T_f$ in the map should be generalized to its deformation and then making the coherent and canonical choice of a suitable ``$\gamma$" in its deformation becomes a highly non-trivial task.

\section*{Acknowledgement}
This article is based on the online lecture series given by the author which is a part of the program ``Elliptic curves and the special values of $L$-functions (ONLINE)'' (2--7 August 2021). The lecture videos are available on \href{https://www.youtube.com/playlist?list=PL04QVxpjcnjhp20VSZQaQK9RnJjY7rp_C}{Youtube}.

The author deeply thanks the organizers Ashay Burungale, Haruzo Hida, Somnath Jha and Ye Tian, the enthusiastic and active participants, as well as the staff members of ICTS.

The author learned a lot about Kato's Euler systems from
 the lecture series given by Masato Kurihara at ``Postech Winter School 2012: The Birch and Swinnerton-Dyer conjecture" (January 2012), and
 from the lecture series given by Kentaro Nakamura and Shanwen Wang at ``An explicit week: on the recent development of Kato's Euler systems and explicit reciprocity laws" (February 2016).
The author is very thankful to them all for teaching him this beautiful subject.

The author was benefited from the detailed comments of Ruichen Xu.
The author also deeply thanks the referee for very detailed reading and for pointing out numerous inaccuracies in an earlier version. It led to significant improvements in the exposition.

\bibliographystyle{amsalpha}
\bibliography{library}

\end{document}